\documentclass[11pt]{amsart}
\usepackage{amsmath,amscd,amssymb}
\usepackage[mathscr]{eucal}
\usepackage{amsmath}
\usepackage{enumerate}
\usepackage{color}

 \newtheorem{thm}{Theorem}[section]
 \newtheorem{cor}[thm]{Corollary}
 \newtheorem{lemma}[thm]{Lemma}
 \newtheorem{prop}[thm]{Proposition}

\theoremstyle{definition}

 \theoremstyle{remark}

\newtheorem{exa}[thm]{Example}
\newtheorem{remark}[thm]{Remark}

\numberwithin{equation}{section}

\def\R{\mathbb{R}}
\def\C{\mathbb{C}}
\def\N{\mathbb{N}}

\title[Rational approximation of semigroups]{Rational approximation of operator semigroups via the $\mathcal B$-calculus}

\author{Alexander Gomilko}
\address{Faculty of Mathematics and Computer Science\\
Nicolas Copernicus University\\
Chopin Str. 12/18\\
87-100 Toru\'n, Poland
}
\email{alex@gomilko.com}
\thanks{}

\author{Yuri Tomilov}
\address{Institute of Mathematics\\
Polish Academy of Sciences\\
\'Sniadeckich 8\\
00-656 Warsaw, Poland
}
\address{Faculty of Mathematics and Computer Science\\
Nicolas Copernicus University\\
Chopin Str. 12/18\\
87-100 Toru\'n, Poland
}

\email{ytomilov@impan.pl}
\thanks{}


\date{\today}


\begin{document}
\begin{abstract}
We improve the classical results by Brenner and Thom\'ee on rational approximations of operator 
semigroups.
In the setting of Hilbert spaces, we introduce a finer regularity scale for initial data, 
provide sharper stability estimates, and obtain
optimal approximation rates. Moreover, we strengthen a result due to Egert-Rozendaal
on subdiagonal Pad\'e approximations of operator semigroups. 
Our approach is direct and based on the theory of the 
$\mathcal B$- functional calculus developed recently.
On the way, we elaborate a new and simple approach to construction of the $\mathcal B$-calculus
thus making the paper essentially self-contained.
\end{abstract}

\maketitle

\section{Introduction}

The theory of rational approximation of $C_0$-semigroups on Banach spaces is a classical chapter of semigroup
theory with a multitude of important applications, e.g. in numerical methods for the study of PDE. 
However, a number of well-known results of the theory were obtained solely in the framework of abstract Banach spaces,
and the influence of geometrical properties of the spaces was not exploited enough.
One of the reasons was apparently a lack of an appropriate functional calculus for semigroup generators taking into account
fine properties of the underlying spaces.
In this paper, using the theory of $\mathcal B$-calculus developed recently in \cite{BaGoTo} and \cite{BaGoTo1}, we intend to fill this gap and to obtain Hilbert space versions of several crucial results
on rational approximation of semigroups.
 
Starting a very brief introduction to the subject, observe that if $-A$ is the generator of a $C_0$-semigroup $(e^{-tA})_{t \ge 0}$ on a Banach space $X,$
then the Cauchy problem 
\begin{equation}\label{cauchy}
x'(t)=-A x(t), \qquad x(0)=x \in {\rm dom}(A),
\end{equation}
where ${\rm dom}(A)$ stands for the domain of $A$, is well-posed.  Many time-discretisation methods of \eqref{cauchy}
lead to the problem of high accuracy approximation of the propagator $(e^{-tA})_{t \ge 0}$ to \eqref{cauchy} by rational functions of $A,$
not relying on the specific structure of $A.$ Without loss of generality, one may assume
that   $(e^{-tA})_{t \ge 0}$ is bounded, so in the sequel we will deal only with bounded semigroups.
Given a sequence of rational functions $(r_n)_{n \ge 1}$ holomorphic in the open right half-plane $\C_+,$
it is thus of substantial interest to characterise the convergence $r_n(tA)x \to e^{-tA}x$ as $ n \to \infty,$ $x \in X,$ in a priori terms,
and to equip this convergence with optimal approximation rates depending
on regularity of the initial data $x.$ 
A simple but a very natural choice is provided by 
\[
r_n(z)= r^n(z/n), \qquad n \ge 1,
\]
for a  rational function $r,$ 
and it will be the main focus
in this paper.

The  study of such approximation problems even in the context  of finite-dimensional $X$
(see e.g. \cite{Hairer}) reveals a necessity to deal with
 $\mathcal A$-stable rational functions and to quantify their approximation properties.
 Recall
that
a rational function $r$ is said to be \emph{$\mathcal A$-stable} if $r$ is holomorphic 
in $\C_+$
 and 
$\sup_{z \in \C_+} |r(z)|\le 1,$
and 
 $r$ is said to be
an approximation of order $q\in \N$ to the exponential function $e^{-z}$  if 
\begin{equation}\label{approxim}
e^{-z}-r(z)=\mbox{O}(|z|^{q+1}),\qquad z\to 0,
\end{equation}
or, in other words, $r^{(k)}(0)=(-1)^k$ for all $0 \le k \le q.$
In the following, we will often refer to $e^{-z}$ as \emph{the exponential}. 

If $r$ is an $\mathcal A$-stable rational approximation to the exponential of certain order, then
by the well-known Lax-Richtmyer-Chernoff theorem
the approximation property
 $\lim_{n \to \infty}r(tA/n)^n x=e^{-tA}x$
for all $x \in X$ and $t \ge 0$ is equivalent to the stability condition $\sup_{n \in\N, t \ge 0}\|r(tA/n)^n\|<\infty,$
see e.g. \cite[p. 437]{Jara1}, and also \cite[Theorem 2.1]{Flory1}.
(This also follows from Theorem \ref{brento} below.)
Moreover, $\mathcal A$-stable approximations $r$ of order $q$  allow one to obtain approximation rates 
for $e^{-tA}x-r^n(tA/n)x$ as $n \to \infty$ depending on $q$, often at the price of considering the initial data $x$ with additional regularity, in particular from  ${\rm dom}(A^s)$ with suitable $s>0,$
and the stability condition plays a role here. Even when the condition does not hold, 
it is often useful to quantify the growth of $\|r(tA/n)^n\|$ for fixed $t > 0.$
A thorough discussion of these and related issues with a number of pertinent references can be found in \cite[Section 2]{Jara1}, see also \cite[Section 1]{Flory} and \cite{Flory1}.

\subsection{Rational approximation of semigroups: a glimpse at the state of the art}

Let ${\rm M}(\R_+)$ be the Banach algebra of bounded Borel measures on $\mathbb R_+=[0,\infty,)$ and let
$\mathcal{LM}(\C_+)$ be the Banach algebra of Laplace transforms $\widehat \nu$ of $\nu \in {\rm M}(\R_+)$
with the norm  $\|\widehat \nu\|_{\mathcal{LM}(\C_+)}:=\|\nu\|_{{\rm M}(\R_+)}.$
Note that if $r$ is an $\mathcal A$-stable rational function,   then $r_n \in \mathcal{LM}(\C_+)$ for every $n \in \mathbb N.$ If, moreover, $r$ is an $\mathcal A$-stable rational approximation to $e^{-z}$
of order $q,$ and 
$\Delta_{n, r, s}(z):=  (e^{-z}-r_n(z))z^{-s},$ then  $\Delta_{n, r, s} \in \mathcal {LM}(\C_+)$ for 
all $n \in \N$ and integers $s$ such that $0 \le s \le q+1.$ 
These observations (see \cite[p. 685]{BT}) put rational approximations $(r_n(A))_{n \ge 1}$ of 
$(e^{-tA})_{t \ge 0}$
into the framework of  the Hille-Phillips (HP-) functional calculus.
Recall that if $-A$ generates a bounded $C_0$-semigroup $(e^{-tA})_{t \ge 0}$ on a Banach space $X,$
and $M:=\sup_{t \ge 0}\|e^{-tA}\|,$
then the Hille-Phillips (HP-) calculus for $A$ is defined as a bounded homomorphism $\Psi_A: \mathcal{LM}(\C_+)\to
L(X)$
given by the following operator version of Laplace transform:
$$\Psi_A(\widehat \nu)x=\widehat \nu (A)x=\int_{0}^\infty e^{-tA}x\, d\nu(t), \qquad  x \in X, \quad \widehat \nu \in \mathcal{LM}(\C_+),$$
so that $\|\hat\nu(A)\|\le M \|\nu\|_{{\rm M}(\R_+)}.$
More details on the HP-calculus can be found e.g. in \cite[Chapter 3.3]{Haase}.
The approach to the study of rational approximations 
by means of the HP-calculus 
was pioneered by Hersch and Kato in \cite{HK},
and elaborated by Brenner and Thomee in \cite{BT}.
In particular, the next basic result in the theory of rational
approximation of semigroups was obtained in 
in \cite[Theorem 1, Theorem 4 and Remarks]{BT}. 
(See also \cite[Sections 2,3]{Bakaev} and \cite[Section 5]{Emamirad}
for generalisations of this result.)
\begin{thm}\label{brento}
Let  $-A$  be the generator of a $C_0$-semigroup $(e^{-tA})_{t\ge 0}$ on Banach space $X$, 
and $M:=\sup_{t \ge 0} \|e^{-tA}\|<\infty.$ 
\begin{itemize}
\item [(i)] If $r$ is an $\mathcal A$-stable rational function, then 
there exists $C=C(r)$ such that for all $n \in \mathbb N$ and $t \ge 0,$
\begin{equation}\label{calley}
\|r^n(tA/n)\|\le CM n^{1/2}.
\end{equation}
\item [(ii)] If  $r$ is an $\mathcal A$-stable rational approximation of order $q \in \N$ to
 the exponential,  and  $s\in \N$ satisfies $s \in (0,q+1], s \neq  (q+1)/2,$ 
then there exists $C=C(s, r)>0$ such that 
\begin{align}\label{bt1}
\|e^{-tA}x-r^n(tA/n)x\|\le& CM \frac{t^s}{n^{s-1/2}} \|A^{s}x\|,\quad s\in (0,(q+1)/2),\\
\|e^{-tA}x-r^n(tA/n)x\|\le& CM \frac{t^s}{n^{sq/(q+1)}}\|A^{s}x\|,\quad s\in ((q+1)/2, q+1],\label{bt2}
\end{align}
for all $x \in {\rm dom}\, (A^s), t \ge 0,$ and $n \in \mathbb N.$
If, under the assumptions above, $s=(q+1)/2,$ then there is  $C=C(s, r)>0$ such that
 \begin{equation}\label{bt3}
\|e^{-tA}x-r^n(tA/n)x\|\le  CM \frac{t^s \log (n+1)}{n^{q/2}}\|A^{s}x\|,
\end{equation}
for all $x \in {\rm dom}\, (A^s), t \ge 0,$ and $n \in \mathbb N.$
\end{itemize}
\end{thm} 
The estimates in Theorem \ref{brento}, are, in general, optimal,  as the examples
of shift semigroups  on $L^1(\mathbb R)$ (or on $C_0(\mathbb R)$) show.
 The sharpness of stability estimate 
\eqref{calley} 
is discussed already in \cite[p. 687]{BT}, see also \cite[Section 2]{Crou}, \cite[Theorem 2.2]{Kovacs} and \cite[Theorem 3.3]{Jara} for more explicit arguments. 
The optimality of \eqref{bt1}, \eqref{bt2} and \eqref{bt3}
is elaborated in \cite[Section 5]{Kovacs}, see also \cite[Remarks, p. 691]{BT}.

It is natural to try to use interpolation and to extend the estimates in
Theorem~\ref{brento} to the whole range of $s \in (0,q+1]$. This idea was
realised in \cite{Kovacs}, where for $s >0$ the Favard interpolation spaces
$\mathcal F_{s}$ containing ${\mathrm{dom}}\, (A^{s})$ were introduced, and
other, related interpolation spaces were studied. However, the estimates
obtained in \cite{Kovacs} do not seem to yield Theorem~\ref{brento} for all
$s \in (0,q+1]$. It looks plausible that one may just repeat the arguments in
\cite[Theorems 1 and 4]{BT} for any fixed $s \in (0, q+1]$, and, using
the product rule for the HP-calculus (as in the proof of Theorem~\ref{mainBT}, (ii) below), obtain Theorem~\ref{brento} as formulated above
with $s$ from $(0,q+1]\cap \mathbb N$ replaced by $s \in (0, q+1]$. A justification of this
claim goes beyond the scope of the present paper, and thus is omitted.

Apart from fundamental use of the HP-calculus, the arguments in \cite{HK}  relied on 
Fourier multiplier estimates for $L^1(\mathbb R_+)$ provided by Carlson's inequality.
A refinement of this technique in \cite{BT} based on  appropriate partitions of unity
and local application of Carlson's inequality led to Theorem \ref{brento}.
These techniques with variations and adjustments
were used in virtually all subsequent papers dealing with rational approximations of 
(in general, non-holomorphic) $C_0$-semigroups.

Comparatively recently, there appeared other approaches to semigroup rational approximation 
relying on Pad\'e approximations to the exponential.
 As emphasized already in \cite{HK} the examples
of rational functions satisfying the assumptions of Theorem \ref{brento} can be constructed
using Pad\'e approximations. Toy examples are provided
by, for instance, Euler's and Crank-Nicolson's schemes corresponding to $r(z)=(1+z)^{-1}$ and $r(z)=(1-z/2)(1+z/2)^{-1}.$
Because of their strong approximation properties and explicit form, Pad\'e approximations play a distinguished role in the theory of rational approximation.
Recall that if $r=P/Q,$ where $P$ and $Q$ are polynomials with  ${\rm deg}\, P=m$ and ${\rm deg}\, Q=n,$ then $r$ 
satisfies \eqref{approxim}
with order $q$ not exceeding $m+n,$
and Pad\'e approximations satisfy \eqref{approxim}
with the maximal order $m+n.$ 
 Setting $Q(0)=1,$
 the Pad\'e approximations $r_{[m,n]}=P_m/Q_n$ of order $m+n$ to the exponential
are unique and can be written explicitly.
So they can be thought of as an infinite $(m,n)$-table.
By a famous result proved in \cite{Ehle} and \cite{Wanner}, see e.g. 
\cite[Theorem 4.12]{Hairer}, the approximations $r_{[m,n]}$ are $\mathcal A$-stable if and only if $m \le n \le m+2.$ 
Thus, the diagonal,  the first subdiagonal, and the second subdiagonal 
Pad\'e approximations $r_{[n,n]}, r_{[n, n+1]}$ and $r_{[n, n+2]},$ $n \in \N,$ 
are of primary interest from the point of view of numerical methods.

In the context of Theorem \ref{brento}, it was observed in \cite{Neu} that when approximating the exponential it is computationally advantageous 
to replace  the sequence $(r(\cdot/n)^n)_{n \ge 1}$ by the sequence  $(r_{[n, n+1]})_{n \ge 1}$ of subdiagonal Pad\'e approximations, so that each $r_{[n,n+1]}$ is a linear combination of simple fractions, and thus $r_{[n,n+1]}(A)$ is a 
linear combination of resolvents.
However, the convergence of  $r_{[n,n+1]}(tA)x$ to $e^{-tA}x$ as $n \to \infty$ for, at least, regular enough data $x$ and all $t \ge 0$
was left  as an open problem in \cite{Neu}, and it was explored very thoroughly in \cite{EgertR} in the spirit of considerations of \cite{BT}.
The key results from \cite{EgertR} can be summarized as follows.
\begin{thm}\label{Egert}
Let $-A$ generate a $C_{0}$-semigroup $(e^{-tA})_{t \geq 0}$ on a Banach space $X,$
and $M:=\sup_{t \ge 0}\|e^{-tA}\|<\infty.$
Let  $(r_{[n, n+1]})_{n \in \N}$ be the first subdiagonal Pad\'e approximations to the exponential,
and let  $s>0$ and $x\in {\rm dom}\, (A^{s})$ be given. 
\begin{itemize}
\item [(i)] If $s >1/2,$ then there exists $C=C(s) >0$ such that 
\begin{equation*}
\|e^{-tA}x - r_{[n,n+1]}(tA)x \|\le C M  \frac{t^{s}}{n^{s-1/2}}  \|A^s x\|
\end{equation*}
for all $t\ge 0,$ and $n\in\mathbb N$ with $n\ge s-\frac{1}{2}$.
\item [(ii)] If $X$ is a Hilbert space, $ a \in (0,s),$ and $(e^{-tA})_{t \geq 0}$ satisfies  
$\|e^{-tA}\|\le Me^{-\omega t}$ for some $M\geq 1$ and $\omega>0,$   then
 there is  $C = C(M, \omega, s-a)$ such that
\begin{align*}
 \left \|e^{-tA}x - r_{[n,n+1]}(tA)x \right \| \leq C \left(\frac{t}{n}\right)^{a} \|A^{s} x\|
\end{align*}
for all $t \geq 0$ and  $n \in \mathbb N$ such that $n \ge \frac{a}{2} - 1$.
\item [(iii)] If $A$ admits a bounded ${\rm H}^{\infty}$-calculus on $\mathbb C_+$ with ${\rm H}^{\infty}$-bound $C$,
then 
\begin{align*}
\left \|e^{-tA}x - r_{[n,n+1]}(tA)x \right \| \leq  2 C \left(\frac {t}{n}\right)^{s} \|A^s x\|
\end{align*}
for all $t \geq 0$ and $n \in \mathbb N$ such that $n > \frac{s}{2} - 1.$
\end{itemize}
\end{thm}
As in the case of Theorem \ref{brento}, the proof of (i) was based on Fourier multiplier estimates and Carlson's inequality,
while the proof of (ii) 
relied on an advanced transference technique developed in \cite{HaaseT} and \cite{HaaseR}.
Note that the assumption of exponential stability in (ii) is indispensable for the methods of \cite{EgertR}.
The assumption of boundedness of the ${\rm H}^\infty$-calculus in (iii) allows one to employ direct estimates
of rational functions in the ${\rm H}^\infty$-norm. 
 It is difficult to express it in abstract terms, especially if $A$ is far from being sectorial. 
For $L^p$ (and more general) spaces $X$ and $C_0$-groups $(e^{-tA})_{t \in \R}$ on $X,$ decaying exponentially as $t \to \infty,$ the stability properties and approximation rates
in Theorems \ref{brento} and \ref{Egert} were improved in \cite{Rozen}. However, the improvements are still  weaker than our Hilbert space results
formulated in the next section.

\subsection{Rational approximation via the $\mathcal B$-calculus: the results}

Recently,  the HP-calculus was used  in \cite{GT} and \cite{GT1}
to unify and extend a number of basic approximation formulas in semigroup theory
and to equip them with optimal approximation rates depending on smoothness 
of initial data. The formulas were recast in terms of completely monotone
and Bernstein functions, and thus
the arguments relied essentially on the study of Laplace transforms of positive measures.
The positivity of measures 
 led to results which are optimal even in the Hilbert space setting.

While the approach in \cite{GT} and \cite{GT1} encompassed some rational functions, e.g. $(1+z)^{-1}$
originating from Euler's formula, the studies of rational approximations for semigroups
often require  to deal with Laplace transforms of, in general, non-positive measures.
Obtaining sharp estimates in this case is problematic,
and the optimal bounds require a finer functional calculus.
The theory of such a calculus, called the $\mathcal B$-calculus in \cite{BaGoTo}, was created recently
 in \cite{BaGoTo} and \cite{BaGoTo1} in the setting of bounded $C_0$-semigroups on Hilbert spaces.
Being based on a Banach algebra $\mathcal B \supset \mathcal{LM}(\C_+),$ the $\mathcal B$-calculus strictly extends the HP-calculus, and leads to, in general,
sharper operator-norm estimates than the ones provided by the HP-calculus.
For more details on the $\mathcal B$-calculus see Section \ref{fcc}.

In the present article, using the $\mathcal{B}$-calculus,
we substantially improve Theorems \ref{brento} and \ref{Egert} if $X$ is a Hilbert space.
Our first result reads as follows.
\begin{thm}\label{mainBT}
Let $-A$ be the generator of a  $C_0$-semigroup $(e^{-tA})_{t\ge 0},$
on a Hilbert space $X,$ and $ M:=\sup_{t \ge 0}\|e^{-tA}\|<\infty.$
\begin{itemize}
\item [(i)] If $r$ is an ${\mathcal A}$-stable rational function, then there exists $C=C(r)$ such that for all $n \in \mathbb N$ and $t \ge 0,$ 
\begin{equation}\label{A220A2}
\|r^n(tA/n)\|\le C \pi M^2 (1+\log(n)).
\end{equation}
 \item [(ii)] If $r$ is an $\mathcal A$-stable rational approximation of order $q$ to the exponential and $s\in (0,q+1],$ then there exists $C=C(s,r)>0$
such that
\begin{equation}\label{AGT}
\|e^{-tA}x-r^n(tA/n)x\|\le C\pi M^2\frac{t^s}{n^{sq/(q+1)}}
\|A^{s}x\|
\end{equation}
for all $x\in {\rm dom}\, (A^s),$  $n \in \mathbb N,$ and $t \ge 0.$
\end{itemize}
\end{thm}
The estimates given in \eqref{AGT} are optimal with respect to approximation rates as
we show in Theorem \ref{CorMa} below. 
The optimality of \eqref{A220A2} is not clear, and it is related to the long-standing open problem
on power boundedness for Cayley's transform of $A,$ see e.g. \cite{BanachG} and Section \ref{fcc} for more on that.

If a rational function $r$ is $\mathcal A$-stable, then the limit 
\[
r(\infty):=\lim_{{\rm Re}\, z \to \infty}r(z),  
\]
exists and $|r(\infty)|\le 1.$ The stability part of Theorem \ref{mainBT} can be strengthened if
one assumes
that $r$ is small at infinity in the sense that
$|r(\infty)|<1.$
The latter smallness condition, called sometimes  strong $\mathcal A$-stability or $L$-stability in the literature, is part of condition (*) in \cite[p. 687]{BT}.
In the situation of Hilbert spaces it suffices to employ this smallness condition alone. 
As we show below the condition ensures  convergence of rational approximations on the whole of $X,$
the property appearing usually in the study of rational approximations for \emph{holomorphic} semigroups.
\begin{thm}\label{mainBTin}
Let $A$ be as in Theorem \ref{mainBT}, and 
let $r$ be 
an ${\mathcal A}$-stable  
rational function satisfying 
\begin{equation}\label{inftyi}
|r(\infty)|<1.
\end{equation}
Then there exists $C=C(r)$ such that
\begin{equation}\label{Ainf}
\|r^n(tA/n)\|\le  C \pi M^2, \qquad n \in \mathbb N, \quad t \ge 0.
\end{equation}
As a consequence,  if $r$ is an $\mathcal A$-stable rational approximation of any order 
to the exponential, and \eqref{inftyi} holds,
then
\begin{equation}\label{conve}
\lim_{n \to \infty} r^n(tA/n)x=e^{-tA}x, \qquad x \in X, \quad  t \ge 0.
\end{equation}
 \end{thm}
Observe that the first and the second subdiagonal Pad\'e approximations to the exponential
provide natural examples of $r$ satisfying the assumptions of Theorem \ref{mainBTin}.

Using the power of the $\mathcal B$-calculus,
we also substantially strengthen Theorem \ref{Egert} 
and obtain an estimate similar to Theorem \ref{Egert},(iii)
for \emph{arbitrary} bounded Hilbert space $C_0$-semigroups,
however modulo a logarithmic correction term.
Namely, the next statement holds.
\begin{thm}\label{mainER}
Let $-A$ be the generator of a $C_0$-semigroup $(e^{-tA})_{t\ge 0}$ on a Hilbert space $X$,
and let $M:=\sup_{t \ge 0}\|e^{-tA}\|<\infty.$
If $(r_{[n,n+1]})_{n \in \N}$ are the first subdiagonal Pad\'e approximations
to the exponential,
then for every $s>0$ there exists $C=C(s)>0$ 
such that 
\begin{equation}\label{mainER1}
\|e^{-tA}x-r_{[n,n+1]}(tA)x\|
\le \frac{\pi M^2 t^s}{n^s}
 \left(C
+4\log(2n+1)
\right)
\|A^s x\|
\end{equation}
for all $x\in {\rm dom}\, (A^s), t \ge 0,$ and $n \in \N$ such that   $n \ge (s-1)/2.$
\end{thm}
It would be of interest to clarify whether the logarithmic correction term  in the right hand side of \eqref{mainER1} can be omitted.
Recalling that the optimality of the logarithmic bound in \eqref{A220A2} has been open for a long time,
this problem does not look surprising.
It is also related to another, apparently open, problem of whether  $r_{[n,n+1]}(tA) \to e^{-tA}$ as $n \to \infty$ strongly for every $t \ge 0$ or, equivalently, whether $\sup_{n \in \N} \|r_{[n,n+1]}(tA)\|<\infty.$

Finishing this section, we fix some relevant notations and conventions.
For a Hilbert space $X,$ we denote by $\langle \cdot, \cdot \rangle$ the inner product of $X$, and we
 let $L(X)$ stand for the Banach algebra of all bounded linear operators on
$X.$ The domain and the spectrum of a linear operator $A$ on $X$ are
denoted by ${\rm dom}(A)$ and $\sigma(A),$  respectively.
If $A$ a densely defined linear operator $A$ on $X,$ then $A^*$ denotes the Hilbert space adjoint of $A.$

Denote by ${\rm Hol}(\Omega)$ the space of functions holomorphic on a domain $\Omega \subset \mathbb C.$
For any function $f: \mathbb C_+ \to \mathbb C$ we write $f(\infty):=\lim_{{\rm Re}\, z \to \infty}f(z)$
if the latter limit exists in $\mathbb C.$

For a subset $S$ of the complex plane $\mathbb C$ we denote by $\partial S$ the topological boundary of $S$,
and by $\overline S$ the closure of $S.$ 
The open right half-plane $\{z \in \mathbb C: {\rm Re}\, z >0\}$ is denoted by $\mathbb C_+.$
Using the notation $x+iy$ (or similar notations)
we always mean that $x={\rm Re}\, z$ and $y={\rm Im}\, z.$

Writing $C=C(r)$ for a rational function $r$ we will mean that $C$ depends on $r$ and
other parameters associated exclusively with $r,$ such as its approximation order $q$ or its domain of 
holomorphicity.
If $C$ will depend on other parameters, then they will be mentioned explicitly.

\section{Function-theoretical estimates}

It was shown in \cite[Sections 2.2 and 2.4]{BaGoTo}, see also \cite[Section 1.5]{Vitse}, 
 that if $\mathcal B$ is the space of functions $f$ holomorphic in 
 $\C_+$ and satisfying 
\begin{equation}  \label{bdef0}
\|f\|_{\mathcal B_0}:=\int_{0}^{\infty} \sup_{y \in \R }|f'(x+iy)| \, dx <\infty,
\end{equation}
then $\mathcal {LM}(\C_+) \subset \mathcal B \subset {\rm H}^\infty(\C_+),$
where ${\rm H}^\infty(\C_+)$ is the Banach algebra of  bounded holomorphic functions on $\C_+.$
Moreover, $\mathcal B$  
equipped with the norm
 $\|f\|_{\mathcal B}:=\|f\|_{{\rm H}^\infty(\C_+)}+ \|f\|_{\mathcal B_0}, f \in \mathcal B$, is a Banach algebra,
and if $\widehat \nu \in \mathcal {LM}(\C_+),$ then $\|\widehat\nu\|_{\mathcal B}\le \|\widehat \nu\|_{\mathcal {LM}(\C_+)}.$
For every $f \in \mathcal B$ the limit  $f(\infty)=\lim_{{\rm Re}\, z \to \infty}f(z)$ exists, $f(\infty) \in \C,$ 
and $\mathcal B_0:=\{f \in \mathcal B: f(\infty)=0\}$ is the closed subalgebra of $\mathcal B.$
Note that \eqref{bdef0} defines a seminorm on $\mathcal B,$  and since by \cite[Proposition 2.2, (2)]{BaGoTo} 
\[
\|f\|_{{\rm H}^\infty(\C_+)}\le |f(\infty)| + \|f\|_{\mathcal B_0}, \qquad f \in \mathcal B,
\]
it defines a norm on $\mathcal B_0$ equivalent to the norm on $\mathcal B.$

Given $s>0$ and an appropriate rational function $r$ or the first subdiagonal Pad\'e approximations $(r_{[n, n+1]})_{n \ge 1}$ to the exponential,
we obtain sharp $\mathcal B_0$-norm estimates
for sequences of the form
$(r(z/n)^n)_{n \ge 1},$  $((e^{-z}-r^n(z/n))z^{-s})_{n \ge 1}$ and $((e^{-z}-r_{[n,n+1]}(z))z^{-s})_{n \ge 1},$
arising 
in the functional calculi approach to Theorems  \ref{mainBT}, \ref{mainBTin} and \ref{mainER}.
Using the $\mathcal B$-calculus,
these estimates will then be transferred to the corresponding 
operator norm-estimates in \eqref{A220A2}, \eqref{AGT}, \eqref{Ainf} and \eqref{mainER1}.
Note that if $f \in \mathcal B,$ $a>0,$ and 
$f_a(z):=f(az),\, z \in \mathbb C_+,$
 then $f_a \in \mathcal B$ and
\begin{equation}\label{scal}
\|f_a \|_{\mathcal B_0}=\|f\|_{\mathcal B_0},  \qquad \|f_a\|_{\mathcal B}=\|f\|_{\mathcal B}.
\end{equation} 
This very simple invariance property will be useful for our studies.

\subsection{Stability estimates in terms of $\mathcal B$-norms}\label{fte}

In this section, given an $\mathcal A$-stable rational function $r,$
we will study the $\mathcal B$-norm estimates for $r^n$ and other rational 
functions mentioned above. Since $|r(\infty)|\le 1$ and $r_{[n,n+1]}(\infty)=0$
it will suffice to concentrate on estimates of their $\mathcal B_0$-(semi)norms.
The next elementary lemma will be crucial.

\begin{lemma}
If a rational function $r$ is $\mathcal A$-stable,
different from constant, and satisfies
$|r(it)|=1$ for every $t \in \mathbb R,$
then $r$ is a finite Blaschke
product:
\begin{equation}\label{bl}
r(z)=c \prod_{j=1}^m \frac{z_j-z}{\overline{z}_j+z},
\end{equation}
for some $(z_j)_{j=1}^m \subset \C_+$ and $c\in \C, |c|=1.$
\end{lemma}
\begin{proof}
Indeed, let $|r|=1$ on $i\mathbb R$, and $r=P/Q$, where $P$ and $Q$ are
polynomials. Then ${\mathrm{deg}}\, P={\mathrm{deg}}\, Q$, and therefore
$\lim _{|z|\to \infty}|r(z)|=1$. Let 
\[
B_{m}(z):=\prod _{j=1}^{m} \frac{z_{j}-z}{\overline{z}_{j}+z}, \qquad z \in \mathbb C_+,
\]
be the Blaschke product
in $\mathbb C_{+}$ whose zeros $(z_j)_{j=1}^m$ are the same as the zeros of $r$ in
$\mathbb C_{+}$, counting multiplicities. Observe that both,
$r/ B_{m}$ and $B_{m}/r$, are holomorphic and bounded in
$\mathbb C_{+}$, and have moduli identically equal to $1$ on
$i\mathbb R$. From the maximum principle it follows that $r/B_{m}$ is a
constant in $\mathbb C_{+}$ (and then in $\mathbb C$) of modulus $1$.
\end{proof}
\noindent
The direct argument above  
can essentially be found e.g. in \cite[Chapter 1.2, p. 6]{Garnett},
and we gave it here in view of its importance for the sequel.

Thus, if a rational function $r$ is $\mathcal A$-stable and different from a constant,
then either $|r| \not \equiv 1$ on $i\R,$ or,
otherwise $r$ is a finite Blaschke product.
We will split the study of powers of $r$ into considering these two cases.
First, we consider the situation when $r$ is a Blaschke product.
This assumption will be replaced by a more general assumption on $r$
to have at least one zero in $\C_+.$
The arguments are essentially the same under both assumptions,
while the technical details become somewhat clearer in a more general set-up. 
We will need a family of auxiliary functions.
 For $n \in \mathbb N$ and $s \ge 0,$ define
\[
F_{n,s}(x,t):=\frac{((x-1)^2+t)^{n-1}}{((x+1)^2+t)^{n+1+s}}, \qquad x,t>0,
\]
and denote
\[
a_{n,s}:=\frac{(2n+s)-2\sqrt{(n+1+s)(n-1)}}{(2+s)}
\]
and
\[
b_{n,s}:=a_{n,s}^{-1}
=\frac{(2n+s)+2\sqrt{(n+1+s)(n-1)}}{(2+s)}.
\]
Observe that
\begin{equation}\label{estb}
1\le \frac{2n+s}{2+s}\le
b_{n,s}\le \frac{2(2n+s)}{2+s}\le 2n.
\end{equation}

The next elementary technical lemma will simplify 
our estimates of powers of rational functions
having at least one zero in $\C_+.$
\begin{lemma}\label{L11}
For $n \in \mathbb N$ and $s \ge 0,$ let
\begin{equation}\label{gns}
G_{n,s}(x):=\sup_{t>0}\,F_{n,s}(x,t),\qquad x>0.
\end{equation}
Then
\begin{equation}\label{gns1}
G_{n,s}(x)=
\begin{cases}
\frac{(x-1)^{2(n-1)}}{(x+1)^{2(n+1+s)}},& \qquad
x\not\in[a_{n,s},b_{n,s}],\\
\frac{(n-1)^{n-1}}{(n+1+s)^{n+1+s}}
\cdot \frac{(2+s)^{2+s}}{4^{2+s}x^{2+s}},&
\qquad x \in[a_{n,s},b_{n,s}],
\end{cases}
\end{equation}
for all $n \in \mathbb N$ and $s \ge 0.$
\end{lemma}

\begin{proof}
If $x>0, s \ge 0,$ and $n \in \mathbb N$ are fixed, then 
\[
\frac{((x+1)^2+t)^{n+2+s}}{(x-1)^2+t)^{n-2}}
\frac{d}{dt}F_{n,s}(x,t)
=-(2+s)t+w_{n,s}(x),
\]
where
\[
w_{n,s}(x)
=-(2+s)x^2+2(2n+s)x-(2+s).
\]
Since the roots of $w_{n,s}(x)=0$ are precisely $a_{n,s}$ and $b_{n,s}$,
we have
$w_{n,s}(x)>0$ if and only if $x\in (a_{n,s},b_{n,s}).$
So, if $x\not\in [a_{n,s},b_{n,s}]$, then
\[
G_{n,s}(x)=
F_{n,s}(x,0)=
\frac{(x-1)^{2(n-1)}}{(x+1)^{2(n+1+s)}}.
\]
On the other hand,
if $x\in[a_{n,s},b_{n,s}]$ and $t_{n,s}(x)={w_{n,s}(x)}(2+s)^{-1},$ then
\[
G_{n,s}(x)=
F_{n,s}(x,t_{n,s}(x))=\frac{(n-1)^{n-1}}{(n+1+s)^{n+1+s}}
\cdot \frac{(2+s)^{2+s}}{4^{2+s}x^{2+s}}. \qedhere
\]
\end{proof}

Recall that if $v$ is the Cayley transform, i.e.
$v(z)=\frac{1-z}{1+z},$
then by \cite[Lemma 3.7]{BaGoTo} (and its proof),
\begin{equation}\label{CJFA1}
\|v^n\|_{\mathcal{B}_0}\le 2+2\log(2n),\qquad n\in \N.
\end{equation}
The logarithmic bound in \eqref{CJFA1} is optimal by \cite[Lemma 5.1]{BaGoTo1}.
For a rational function $r$ with at least one zero in $\C_+,$ our estimates 
will be reduced to the particular case when $r=v.$ For $s \ge 0$ let
\begin{equation}\label{defeta}
\eta_s(z):=(1+z)^{-s}, \qquad z \in \C_+.
\end{equation}
\begin{lemma}\label{StabLP}
Let $r$ be an $\mathcal A$-stable rational function on $\mathbb C_+$
satisfying $r(\lambda)=0$ for some $\lambda \in \C_+,$
and let $s>0.$
Then  there exist $C=C(r)>0$ and $C_1=C_1(r,s)>0$ such that
\begin{equation}\label{Cgen1}
\|r^{n}\|_{\mathcal{B}_0}\le C (1+\log(n))
\end{equation}
and 
\begin{equation}\label{FolP1}
\|r^{n} \eta_s\|_{\mathcal{B}_0}\le C_1(r,s)
\end{equation}
for all $n \in \N.$ 
\end{lemma}

\begin{proof}
Let $n \in \mathbb N$ be fixed. First observe that
since $r(\lambda)=0,$ 
the $\mathcal A$-stability of $r$ and the maximum principle imply that
\begin{equation}\label{lambda}
r(z)=\left(\frac{\lambda-z}{\overline{\lambda}+z}\right)r_0(z),\qquad z\in \C_{+}, 
\end{equation}
where a rational function $r_0$ satisfy $\|r_0\|_{{\rm H}^\infty(\C_{+})}\le 1.$ Clearly,
$r_0=r(\infty)+ P/Q,$ where $P$ and $Q$ are polynomials
with
${\rm deg} \, P < {\rm deg}\, Q,$
and $r(\infty)=\lim_{z \to \infty}r(z).$
So, by a partial fraction expansion, there exists $C=C(r_0)>0$ such that 
\begin{equation}\label{lambda1}
|r'_0(z)|\le \frac{C}{|\overline{\lambda}+z|^2}, \qquad z\in \C_{+}.
\end{equation}

Hence, using (\ref{lambda}), (\ref{lambda1}) and $\|r_0\|_{{\rm H}^\infty(\C_+)}\le 1,$
\begin{equation}\label{wn}
|(r^n(z))'|\le n \frac{|\lambda-z|^{n-1}}
{|\overline{\lambda}+z|^{n-1}}
\frac{C + 2{\rm Re}\lambda}{|\overline{\lambda}+z|^2},\qquad z\in \C_{+}.
\end{equation}

Denoting $C_\lambda=C+2{\rm Re}\,\lambda,$
$\lambda=a+ib$, and $z=x+iy,$ 
we infer from \eqref{wn} that
\begin{equation}\label{CJFA}
\begin{aligned}
\|r^n\|_{\mathcal{B}_0}
\le & n C_\lambda \int_0^\infty \sup_{y\in \R} \,
\frac{((x-a)^2+(y+b)^2)^{(n-1)/2}}
{((x+a)^2+(y+b)^2)^{(n+1)/2}}\,dx\\
=& n C_\lambda \int_0^\infty \sup_{y\in \R} \,
\frac{((x-1)^2+y^2)^{(n-1)/2}}
{((x+1)^2+y^2)^{(n+1)/2}}\,dx
 \\
=&
\frac{C_\lambda}{2}\,\|v^n\|_{\mathcal{B}_0}.
\end{aligned}
\end{equation}
Hence (\ref{CJFA}) and  \eqref{CJFA1} yield \eqref{Cgen1}.

Next, let $s >0$ be fixed. By (\ref{wn}) and the $\mathcal A$-stability of $r$, 
\begin{equation}\label{derivat}
\begin{aligned}
|(r^n(z)\eta_s(z))'|\le&  \frac{|(r^n(z))'|}{|1+z|^s}+
s \frac{|r^n(z)|}{|1+z|^{s+1}}\\
\le& n \frac{|\lambda-z|^{n-1}}
{|\overline{\lambda}+z|^{n-1}}
\frac{C_\lambda}{|\overline{\lambda}+z|^{2+s}}+
\frac{s}{|1+z|^{s+1}}, \qquad z \in \mathbb C_+.\notag
\end{aligned}
\end{equation}
Since
\begin{equation}\label{Jeps}
\int_0^\infty \sup_{y>0}\,\frac{s}{((x+1)^2+y^2)^{(s+1)/2}}
\,dx=\int_0^\infty \frac{s\,dx}{(1+x)^{s+1}}=1,
\end{equation}
we have
\[
\|r^n\eta_s\|_{\mathcal{B}_0}\le  C_\lambda  I_{n,s}+1,
\]
where
\begin{align*}
I_{n,s} =& n \int_0^\infty \sup_{y\in \R} \,
\frac{((x-a)^2+(y+b)^2)^{(n-1)/2}}
{((x+a)^2+(y+b)^2)^{(n+1+s)/2}}\,dx\\
=& n\int_0^\infty \sup_{y\in \R} \,
\frac{((x-a)^2+y^2)^{(n-1)/2}}
{((x+a)^2+y^2)^{(n+1+s)/2}}\,dx\\
=&n\int_0^\infty (G_{n,s}(x))^{1/2}\,dx
\end{align*}
and $G_{n,s}$ is given by \eqref{gns}.

Thus, by \eqref{gns1},
\begin{align*}
I_{n,s}
=&n\int_0^{a_{n,s}}\frac{(1-x)^{n-1}}{(x+1)^{n+1}}\,dx
+n\int_{b_{n,s}}^\infty\frac{(x-1)^{n-1}}{(x+1)^{n+1}}\,dx\\
+&n\frac{(n-1)^{(n-1)/2}}{(n+1+s)^{(n+1+s)/2}}\frac{(2+s)^{1+s/2}}
{2^{2+s}}
\int_{a_{n,s}}^{b_{n,s}}
\frac{dx}{x^{1+s/2}}.
\end{align*}

Recalling that $a_{n,s}=b^{-1}_{n, s},$ observe that
\[
n\int_0^{a_{n,s}}\frac{(1-x)^{n-1}}{(x+1)^{n+1}}\,dx
+n\int_{b_{n,s}}^\infty\frac{(x-1)^{n-1}}{(x+1)^{n+1}}\,dx=
2\int_{b_{n,s}}^\infty d\left(\frac{(x-1)^{n}}{(x+1)^{n}}\right)\le 2.
\]
Therefore, taking into account (\ref{estb}), 
\begin{align*}
I_{n,s} \le 2+\frac{2}{n^{s/2}}\frac{(2+s)^{1+s/2}}
{2^{2+s}s}b_{n,s}^{s/2}
\le 2+\frac{(1+s/2)^{s/2+1}}{s},
\end{align*}
and then
\begin{equation}\label{FolP}
\|r^n \eta_s\|_{\mathcal{B}_0}\le
\left(2+ \frac{(1+s/2)^{s/2+1}}{s} \right)C_\lambda+1.
\end{equation}
Thus, \eqref{FolP1} holds with $C(r,s)$ given by the right-hand side of \eqref{FolP}. $\qedhere$
\end{proof}

Now we turn to the case when $|r|$ is not identically $1$ on $i\R$ (but $r$ may not have zeros in $\C_+$).
As above, we will first need a technical lemma facilitating
our estimates of $r^n, n \in \mathbb N,$ under this assumption.

For $\alpha>0,$  $s\ge 0,$ and $n \in \mathbb N$ define 
\[
S_{\alpha, n,s}(x,t):=\frac{e^{-\alpha n x/(x^2+t)}}{(x^2+t)^{1+s/2}}, \qquad x, t>0.
\]
\begin{lemma}\label{S11}
For all  $\alpha, \beta>0$ and $s\ge 0$ there are 
$C=C(\alpha,\beta,s)>0$
such that for every $n\in \N,$
\begin{equation}\label{S111}
\int_{\beta/n}^\infty \sup_{t>0}\,S_{\alpha, n, s}(x,t)\,dx\le 
\begin{cases} C\frac{\log(n+1)}{n},& \,\, \text{if}\,\, s=0,\\
\frac{C}{n},& \,\, \text{if}\,\, s>0.
\end{cases}
\end{equation}
\end{lemma}

\begin{proof}
Let $\alpha, \beta >0$ and $s\ge 0$ be fixed. Fix also $n \in \mathbb N$ and $x>0.$ Then for all $t >0,$
\[
e^{-\alpha n x/(x^2+t)}(x^2+t)^{3+s/2}\frac{d}{dt}S_{\alpha, n,s}(x, t)
=an x-(1+s/2)(x^2+t),
\]
and $\frac{d}{dt}S_{\alpha, n,s}(x,t)$ has the unique zero at  $t_0=\alpha_snx-x^2,$
where we denoted $\alpha_s=\frac{\alpha}{1+s/2}$ for shorthand. 
Therefore,
\[
\sup_{t>0}\,S_{\alpha, n,s}(x,t)=S_{\alpha, n,s}(x,t_0)=\frac{e^{-(2+s)/2}}{(\alpha_s nx)^{1+s/2}},\qquad \text{if}\,\,  x\in (0,\alpha_s n),
\]
and
\[
\sup_{t>0}\,S_{\alpha, n,s}(x,t)=S_{n,s}(x,0)
=\frac{e^{-\alpha n/x}}{x^{2+s}},\qquad \text{if} \,\, x \in (\alpha_s n, \infty).
\]
Then for large enough $n,$
we have
\[
\int_{\beta/n}^\infty \sup_{t>0}\,S_{\alpha, n,0}(x,t)\,dx\le \frac{1}{\alpha_s n}\int_{\beta/n}^{\alpha n}\frac{dx}{x}+
\int_{\alpha n}^\infty \frac{dx}{x^2}\le C\frac{\log(n+1)}{n}
\]
and, if $s>0$,
\[
\int_{\beta/n}^\infty \sup_{t>0}\,S_{\alpha, n,s}(x,t)\,dx\le \frac{1}{(\alpha_sn)^{1+s/2}}
\int_{\beta/n}^{\alpha_sn}\frac{dx}{x^{1+s/2}}
+\int_{\alpha_sn}^\infty \frac{dx}{x^{2+s}}\le \frac{C}{n},
\]
for appropriate constants $C=C(\alpha,\beta,s) >0.$
This finishes the proof.
\end{proof}

The next lemma 
is a counterpart of Lemma \ref{StabLP} 
under complementary  assumptions on $r$.
For $R>0$ denote
\[
\mathbb D_R^{+}:=\{z\in \C_{+}:\,|z|< R\} \qquad \text{and}\qquad \mathbb D_R:=\{z \in \mathbb C: |z|<R\}.
\]
\begin{lemma}\label{NC1}
Let $r$ be an $\mathcal A$-stable rational function
satisfying $|r|\not \equiv 1$ on $i\R,$
and let $s >0.$
Then  
there exist $C=C(r)>0$  and $C_1=C_1(r,s)$ such that
\eqref{Cgen1} and \eqref{FolP1} hold
for all $n \in \N.$
\end{lemma}

\begin{proof}
Since  $|r|$ is not identically $1$ on $i\R,$ 
the set $\{z\in i\mathbb R: |r(z)|=1\}$ is at most finite.   
Therefore, taking into account that $r$ is $\mathcal A$-stable,  there exists  $R>0$ such that
$
|r(\pm iR)|<1.
$
Since, by the maximum principle,
$|r(z)|<1$ for all $ z \in \C_+,$  we 
have
\[
\omega:=\sup_{z \in \C_+, |z|=R}|r(z)| <1.
\]
If $\alpha=\frac{|\log\omega|}{R}$ and 
\[
f(z):=e^{\alpha z}r(z), \qquad z \in \overline{\C}_+, 
\]
then
$$
|f(z)|\le 1, \qquad z \in \partial \mathbb D_R^{+},
$$
and the maximum principle yields 
$|f(z)|\le 1$ for all $z\in \mathbb D_R^{+}.$
Thus
\begin{equation}\label{Ep11}
|r(z)|\le e^{- \alpha {\rm Re}\,z},\qquad z\in \mathbb D_R^{+}.
\end{equation}
By applying the argument above to $r^*(z):=r(1/z)$
in place of $r,$
we infer that
\begin{equation}\label{Ep12}
|r(z)|\le e^{-\alpha {\rm Re}\,(1/z)},\qquad z\in \C_{+}\setminus \mathbb D_{R}^{+}.
\end{equation}
Moreover, as in the proof of Lemma \ref{StabLP},
\begin{equation}\label{eqq}
|r'(z)|\le \frac{C}{|1+z|^2}, \qquad z \in \C_+,
\end{equation}
for some $C=C(r)>0.$

Thus, for every $n \in \N$,  
 the estimates (\ref{Ep11}), (\ref{Ep12}) and \eqref{eqq} combined with   the $\mathcal A$-stability of $r$ imply that
\begin{equation}\label{estimm}
|(r^n(z))'|\le\begin{cases}  C n e^{-\alpha (n-1){\rm Re}\,z},& \qquad z\in \mathbb D_{R}^{+},\\
 C n e^{-\alpha (n-1){\rm Re}\,(1/z)}|1+z|^{-2},& \qquad z\in \C_{+}\setminus \mathbb D_{R}^{+}.
\end{cases}
\end{equation}
We simultaneously obtain both bounds \eqref{Cgen1} and \eqref{FolP1} using the fact that $r^n=r^n\eta_0.$
Fix $s \ge 0$ and  $n \ge 2,$ and let $z=x+iy\in \C_{+}.$ Writing 
\begin{equation}\label{derivat1}
 (r^n\eta_s(z))'=\frac{(r^n(z))'}{(z+1)^s}
-s\frac{r^n(z)}{(z+1)^{s+1}},\qquad z \in \C_+,
\end{equation}
and employing \eqref{estimm}, 
 we have
\begin{equation*}
\begin{aligned}
\|r^n\eta_s\|_{\mathcal{B}_0}\le&\int_0^{\frac{R}{n}}\sup_{y\in \R}\,\frac{|(r^n(z))'|}{|1+z|^s}\,dx
+\int_{\frac{R}{n}}^\infty\sup_{|z|>R,\,y\in \R}\,\frac{|(r^n(z))'|}{|1+z|^s}\,dx\\
+&\int_{\frac{R}{n}}^{R}\sup_{|z|<R,\,y\in \R}\,\frac{|(r^n(z))'|}{|1+z|^s}\,dx
+s\int_0^\infty \sup_{y\in \R}\,\frac{|r^n(z)|}{|1+z|^{s+1}}\,dx\\
\le& C+Cn\int_{\frac{R}{n}}^\infty \sup_{t>0}\,S_{\alpha, n-1,s}(x,t)\,dx
+Cn\int_{\frac{R}{n}}^R\frac{e^{-\alpha (n-1)x}}{(x+1)^{2+s}}\,dx\\
+s&\int_0^\infty \frac{dx}{(x^2+1)^{(1+s)/2}}
\le C_1+Cn\int_{\frac{R}{n}}^\infty \sup_{t>0}\,S_{\alpha, n-1,s}(x,t)\,dx
\end{aligned}
\end{equation*}
for some $C_1=C_1(s)>0.$
The estimates 
(\ref{Cgen1}) and (\ref{FolP1}) follow from  (\ref{S111}).
\end{proof}

Observe that the estimate given by Lemma \ref{NC1} is optimal, as the following example shows.
\begin{exa}\label{SE}
Let $z=x+iy \in \overline{\C}_+,$ and consider the rational function
\[
r(z)=\frac{(z+2)^2}{(z+1)(z+4)}.
\]
Noting that
\[
r'(z)=\frac{z^2-4}{(z+1)^2(z+4)^2},
\]
by a simple but somewhat tedious calculation, we conclude that 
\[
|r(z)|\ge \frac{|z-1|}{|z+1|},\qquad
|r'(z)|\ge\frac{1}{4|z+1|^2}, \quad {\rm Re}\, z\ge 8.
\]
Hence, for all $n\ge 2,$ 
\[
|(r^n(z))'|\ge \frac{n}{4}\frac{|z-1|^{n-1}}{|z+1|^{n+1}}, \qquad {\rm Re}\, z\ge 8.
\]
so that, by Lemma \ref{StabLP}  with $s=0,$ 
\[
\|r^n\|_{\mathcal{B}_0}\ge \int_8^n\sup_{y\in\R}\,|(r^n(z))'|\,dx \ge
\frac{n(n-1)^{(n-1)/2}}{8(n+1)^{(n+1)/2}}\int_8^n\frac{dx}{x}
\ge C\log(n+1)
\]
for some $C>0.$

Moreover, for all $t\in \R,$
\[
|r(it)|^2=\frac{(t^2+4)^2}{(t^2+1)(t^2+16)}\le 1,
\]
and
\[
|r(it)|<1, \qquad t \neq 0,\quad r(0)=1, \quad \text{and} \quad r(\infty)=1.
\]
\end{exa}

Combining now Lemmas \ref{StabLP} and \ref{NC1},
we obtain \eqref{Cgen1} and \eqref{FolP1} for any $\mathcal A$-stable $r.$ 
\begin{cor}\label{MAC}
Let $r$ be an $\mathcal A$-stable rational function, and let $s>0$.
Then  
there exist $C=C(r)>0$  and $C_1=C_1(r,s)>0$ such that
\eqref{Cgen1} and \eqref{FolP1} hold for all $n \in \N.$ 
Moreover, if $r_n(z)=r^n(z/n),$ $z  \in \C_+,$
then 
\begin{equation}\label{FolP2}
\|r_n\eta_s\|_{\mathcal B_0}\le C_1(r,s) +1
\end{equation}
for all $n \in \N.$ 
\end{cor}

\begin{proof}
If $r$ is a constant function, then the statement is obvious.
Otherwise, as noted in the beginning of this section, either 
$|r|$ is not identically $1$ on $i\R,$ or $r$ is a finite Blaschke
product.
Thus, by combining Lemmas \ref{StabLP} and \ref{NC1},
we obtain \eqref{Cgen1} and \eqref{FolP1}  for 
some $C=C(r)>0,$ $C_1=C_1(r,s)>0$ and any $\mathcal A$-stable rational function $r.$ 
 
To obtain \eqref{FolP2}
it suffices to note that by \eqref{derivat1} and \eqref{Jeps},
letting $z=x+iy\in \mathbb C_+,$
\begin{align*}
\|r_{n}\eta_s\|_{\mathcal{B}_0}\le& 1+
\int_0^\infty \sup_{y\in \R}\,
\frac{|(r_n(z))'|}{|1+z|^s}\,dx\\
=&1+\int_0^\infty \sup_{y\in \R}\,
\frac{|(r^n(z))'|}{|1+nz|^s}\,dx
\le 1+\int_0^\infty \sup_{y\in \R}\,
\frac{|(r^n(z))'|}{|1+z|^s}\,dx\\
=& 1+ \|r^n \eta_s\|_{{\mathcal B}_0}.  \qedhere
\end{align*}
\end{proof}

The bound \eqref{Cgen1} 
 can be substantially strengthened if $|r(\infty)|<1.$
\begin{lemma}\label{infinity}
Let $r$ be an $\mathcal A$-stable rational function such that $|r(\infty)|<1.$
Then
\begin{equation}\label{R2}
\sup_{n \in \N}\|r^n\|_{\mathcal{B}_0}<\infty.
\end{equation}
\end{lemma}

\begin{proof}
Let $n \in \mathbb N$ be fixed.
Observe that there exists $C=C(r)>0$ such that
\begin{equation}\label{rr}
|(r^n(z))'|
\le \frac{Cn}{|1+z|^2}|r(z)|^{n-1}, \qquad z \in \C_+.
\end{equation}
Next,
fix $\omega \in (|r(\infty)|,1).$ By assumption, there exists $R>0$ such that
\begin{equation}\label{R0}
|r(z)|\le \omega<1, \qquad z \in \overline{\C}_+ \setminus \mathbb D_R^+.
\end{equation}
Arguing as in the proof of Lemma \ref{NC1}
and letting $\alpha=\frac{|\log\omega|}{R},$ 
we infer that by 
the maximum principle, 
\begin{equation}\label{R1}
|r(z)|\le e^{- \alpha {\rm Re}\,z},\qquad z\in \mathbb D_R^{+}.
\end{equation}

From \eqref{rr}, \eqref{R0} and \eqref{R1}, 
taking into account that $\|r\|_{{\rm H}^\infty(\C_+)}\le 1,$
it then follows that 
\[
|(r^n(z))'|\le \begin{cases}  C n e^{-\alpha (n-1) {\rm Re}\,z},& \quad z\in \mathbb D_R^{+},\\
{C n\omega^{n-1}}{|1+z|^{-2}},& \quad z\in \C_{+}\setminus \mathbb D_R^{+}.
\end{cases}
\]
Hence, letting $z=x+i y\in \C_{+},$ we have
\begin{align*}
\|r^n\|_{\mathcal{B}_0}\le& \int_0^R \sup_{y\in\R,\,z\in \mathbb D_{R}^{+}}\,|(r^n(z))'|\,dx+
\int_0^\infty \sup_{y\in\R,\, z\in \C_+\setminus \mathbb D^+_R}\,|(r^n(z))'|\,dx\\
\le& C n\int_0^R  e^{-\alpha (n-1) x}\,dx+
C n \omega^{n-1}\int_0^\infty \frac{dx}{(1+x)^2}\\
\le& \frac{C n}{\alpha (n-1)}+C n\omega^{n-1}
\le \frac{2 C R}{|\log\omega|}+\frac{C}{e\omega|\log\omega|},
\end{align*}
so that \eqref{R2} holds. 
\end{proof}

\subsection{Estimates for approximation rates in terms of $\mathcal B_0$-norms}
If a rational function $r$ approximates the exponential with some order, $s \in [0,q+1],$ and $R_0>0$ is small enough, 
then $r$ extends holomorphically to $\mathbb D_{R_0},$
and for  $n \in \mathbb N$ we let
\[
\Delta_{r, n}(z):=e^{-z}-r^n(z/n), \qquad z \in \mathbb D_{R_0n},
\] 
and
\[ 
  \Delta_{r, n, s}(z):=z^{-s}\Delta_{r,n}(z),  \qquad z\in \mathbb D^+_{R_0n},
\]
so that $\Delta_{r, n, 0}=\Delta_{r, n}$ on $\mathbb D^+_{R_0n}.$
 Clearly,
$\Delta _{r, n} \in {\mathrm{Hol}}(\mathbb D_{R_0n})$  and $ \Delta _{r, n, s} \in {\mathrm{Hol}}(\mathbb D^+_{R_{0}n}).$ Moreover, $\Delta _{r,n,s}$ extends continuously to
$(-iR_{0}n,\allowbreak iR_{0}n)$.
If, in addition, $r$ is $\mathcal A$-stable, then both $\Delta_{r,n}$ and $\Delta_{r,n,s}$ 
extend holomorphically to $\mathbb C_+,$ and continuously to $\overline{\C}_+.$
We denote the extensions by \emph{the same symbols}.

Given a rational approximation $r$ to the exponential, 
we turn to obtaining estimates for the $\mathcal B_0$-norms of $\Delta_{r,n,s}.$
In Section \ref{fcc}, these estimates  
will be converted into semigroup approximation rates by means of the $\mathcal B$-calculus.
The next lemma is an important tool for obtaining 
the $\mathcal B_0$-norm bounds for $\Delta_{r,n,s}$
which are optimal with respect to approximation order of $r$.
\begin{lemma}\label{rates}
Let $r$ be a rational approximation of order $q \in \N$ to the exponential.
If 
\begin{equation}\label{deff}
s\in [0,q+1], \quad \, \delta:=q/(q+1), \quad \, \text{and} \quad \, 
a:=\frac{r^{(q+1)}(0)-(-1)^{q+1}}{(q+1)!},
\end{equation}
then there exist $R=R(r) \in (0,1)$ and $c=c(r)>0$ such that 
the following hold.
\begin{itemize}
\item [(i)] One has 
\begin{equation}\label{main112}
|\Delta_{r, n, s}'(z)|\le
\left(|a|(q+1-s)\frac{|z|^{q-s}}{n^q}+ \frac{c}{n^{\delta s}}\right)
e^{-{\rm Re}\,z} 
\end{equation}
for all $z \in \mathbb D^+_{R n^{\delta}}$ and $n \in \N.$
\item [(ii)]
For every $n \in \mathbb N$ there exists
$u_n
\in {\rm Hol}(\mathbb D_{R n})$ 
such that 
\begin{equation}\label{Ma1}
\Delta_{r, n, s}(z)=
\left(1-e^{az^{q+1}/n^q}e^{u_n(z)}\right){e^{-z}}{z^{-s}}, \qquad |u_n(z)|\le c\frac{|z|^{q+2}}{n^{q+1}},
\end{equation}
for all $z\in \overline{\mathbb D}^+_{Rn^{\delta}}$ and $n \in \N.$ 
\end{itemize}
\end{lemma}
The proof of Lemma \ref{rates} is rather technical, and is postponed to Appendix.

The following result
 is behind Theorem \ref{mainBT},(ii).
\begin{lemma}\label{T1P}
Let $r$ be an $\mathcal A$-stable rational
approximation of order $q$ to the exponential.
If  $\delta=q/(q+1),$ then  for every $s\in (0,q+1]$ 
there exists $C=C(r,s)>0$ such that
\begin{equation}\label{ThAP}
\|\Delta_{r, n, s}\|_{{\mathcal B}_0}\le C n^{-{\delta s}}, \qquad n \in \mathbb \N.
\end{equation}
\end{lemma}

\begin{proof}
Let $s\in (0,q+1]$ and $n \in \N$ be fixed.
From  Lemma \ref{rates},(i)
it follows that there are $R \in (0,1)$ and $c >0$
(both depending on $r$)
such that
\begin{equation}\label{mainP}
|\Delta_{r, n, s}'(z)|\le |a| (q+1-s)\frac{|z|^{q-s}}{n^q}+\frac{c}{n^{\delta s}}e^{-{\rm Re}\,z},
\qquad z\in \mathbb D^{+}_{Rn^{\delta}},
\end{equation}
where $a$ is defined as in \eqref{deff}.
For $z \in \C_+$ write $z=x+iy.$
We have
\begin{equation}\label{bound}
\|\Delta_{r, n, s}\|_{{\mathcal B}_0}\le
B_{n,s}+U^{(1)}_{n,s}+U^{(2)}_{n,s},
\end{equation}
where 
\begin{align*}
B_{n,s}=&\int_0^{n^\delta}
\sup_{y\in \R,\,|z|\le n^{\delta}}\,
|\Delta_{r, n,s}'(z)|\,dx,\\
U^{(1)}_{n,s}=&\int_0^\infty
\sup_{y\in \R, \, |z|>R n^{\delta}}\,|\Delta_{r,n}'(z)||z|^{-s}\,dx,
\end{align*}
and
\[
U^{(2)}_{n,s}=s\int_0^\infty
\sup_{y\in \R,\,|z|>Rn^{\delta}}\,|\Delta_{r,n}(z)|
|z|^{-s-1}\,dx.
\]
We estimate each of the terms in the right hand side of \eqref{bound} separately.

First, note that by \eqref{mainP},
\[
B_{n,s}\le |a| B_{n,s}^{(1)}+ c B_{n,s}^{(2)},
\]
where 
\begin{align*}
&B_{n,s}^{(1)}
=\frac{(q+1-s)}{n^{q}}\int_0^{n^\delta}
\,\sup_{|y|\le n^{\delta}}\,
(x^2+y^2)^{(q-s)/2}\,dx
\,\,\,\,\, \text{if} s\,\, \in (0,q+1),\\
&B_{n,q+1}^{(1)}=0,
\end{align*}
and
\[
B_{n,s}^{(2)}=
\frac{1}{n^{\delta s}}\int_0^{n^{\delta}} e^{-x}\,dx \le\frac{1}{n^{\delta s}}.
\]

\noindent
Furthermore, if $s\in (0,q],$ then
\begin{align*}
B_{n,s}^{(1)}=&\frac{(q+1-s)}{n^{q}}\int_0^{n^\delta}
(x^2+n^{2\delta})^{(q-s)/2}\,dx\\
= &\frac{(q+1-s)}{n^{\delta s}}\int_0^1
(t^2+1)^{(q-s)/2}\,dt
\le \frac{q+1}{n^{\delta s}} 2^{q/2},
\end{align*}
and, if $s\in (q,q+1),$ then
\[
B_{n,s}^{(1)}
= \frac{(q+1-s)}{n^{q}}\int_0^{n^{\delta}}\,
\frac{dx}{x^{s-q}}
= \frac{1}{n^{\delta s}}.
\]

\noindent
So, combining the estimates for $B^{(1)}_{n,s}$ and $B^{(2)}_{n,s}$ above,
we obtain that
\begin{equation}\label{LP}
B_{n,s}\le
 \frac{|a|(q+1)2^{q/2}+c}{n^{\delta s}}.
\end{equation}

Next,
since $r$ is $\mathcal A$-stable, we have $\|\Delta_{r, n}\|_{{\rm H}^\infty(\C_+)}\le 2$.
Using this we infer that
\begin{equation}\label{I1P}
\begin{aligned}
U^{(2)}_{n,s}\le&
2s\int_{Rn^{\delta}}^\infty
\sup_{y\in \R}\,|z|^{-s-1}\,dx+
2s \int_0^{R n^{\delta}}
\sup_{|y|\ge R n^{\delta}}\,|z|^{-s-1}\,dx \\
=&2s\int_{R n^{\delta}}^\infty
\frac{dx}{x^{s+1}}+
2s\int_0^{R n^{\delta}}
\frac{dx}{(x^2+R^2n^{2\delta})^{(s+1)/2}}
\le \frac{2(2+q)}{R^s n^{\delta s}}
\end{aligned}
\end{equation}

\noindent
Finally, to estimate $U^{(1)}_{n,s},$ note that 
\begin{align*}
U^{(1)}_{n,s}
\le& \int_0^\infty
\sup_{y\in \R, \, |z|\ge Rn^{\delta}}|z|^{-s} e^{-x}\,dx
+\int_0^\infty
\sup_{y\in \R, \, |z|\ge Rn^{\delta}}\,
\frac{|(r^n(z/n))'|}{|z|^s}\,dx\\
\le& \frac{1}{R^s n^{\delta s}}
+\int_0^\infty
\sup_{y\in \R, \, |z|\ge Rn^{\delta-1}}\,
\frac{|(r^n(z))'|}{n^s|z|^s}\,dx.
\end{align*}
If
$|z|\ge  Rn^{\delta-1},$
then, recalling that $\delta \in (0,1)$ and $R\in (0,1),$ we have
\begin{align*}
2n|z|\ge
n|z|+R n^\delta \ge Rn^\delta(|z|+1)
\ge R n^{\delta}|z+1|.
\end{align*}
 
\noindent
Hence, if  $C_1(r,s)$ is the constant given by  Corollary \ref{MAC},
then
\begin{equation}\label{I2P}
U^{(1)}_{n,s} \le \frac{1}{R^s n^{\delta s}} +
\frac{2^s}{R^s n^{\delta s}}\int_0^\infty\,\sup_{y\in \R}\,
\frac{|(r^n(z))'|}{|1+z|^s}\,dx
\le \frac{1+2^s C_1(r,s)}{R^s n^{\delta s}}.
\end{equation}

Using \eqref{bound} and taking into account  (\ref{LP}), (\ref{I1P}), 
and (\ref{I2P}),
one obtains (\ref{ThAP}).
\end{proof}
\begin{remark}
By combining  \eqref{main112} and the estimate for $C_1(r,s)$ from the proof of  
Corollary \ref{MAC} (i.e. the proofs of Lemmas \ref{StabLP} and \ref{NC1}), it is easy to show that
$C=C(r,s)$ in \eqref{ThAP} satisfies $\sup_{x\le s \le q+1}C(r,s)<\infty$ for every $x >0.$
\end{remark}

\subsection{Estimates for subdiagonal Pad\'e approximations}

Following the ideology of the previous subsections and having in mind Theorem \ref{mainER},
we proceed with obtaining fine $\mathcal B_0$-norm estimates for $z \to (e^{-z}- r_{[n,n+1]}(z))z^{-s},$
where $(r_{[n,n+1]})_{n \ge 1}$ are the first subdiagonal Pad\'e approximations of
$e^{-z}.$

Recall that if $P_n$ and $Q_{n+1}$ are polynomials such that  $r_{[n,n+1]}=P_n/Q_{n+1},$ $ n \in \N,$ 
$Q_{n+1}(0)=1,$
then for all $n \in \N,$
\begin{align}\label{polynom}
P_n(z) =& \sum_{j=0}^{n} \frac{(2n+1-j)!n!}{(2n+1)!j!(n-j)!} (-z)^j, \\
Q_{n+1}(z) =& \sum_{j=0}^{n+1} \frac{(2n+1-j)!(n+1)!}{(2n+1)!j!(n+1-j)!} z^j,\label{polynom1}
\end{align}
and $r_{[n,n+1]}$ are 
  $\mathcal A$-stable 
	rational approximations of order $2n+1$ to $e^{-z},$
see e.g. \cite[Theorem 3.11]{Hairer}.
Fix 
 $n \in \N$ and define the approximation error $\Delta_{r_{[n,n+1]}}$ by 
$$
\Delta_{r_{[n,n+1]}}(z):=r_{[n,n+1]}(z)-e^{-z}, \qquad z \in \C_+.
$$
The next classical Perron representation for  $\Delta_{r_{[n,n+1]}}:$
\begin{equation}\label{AppDelta}
\Delta_{r_{[n,n+1]}}(z)=\frac{1}{Q_{n+1}(z)}\frac{z^{2n+2}}{(2n+1)!}
\int_0^1 (1-t)^n t^{n+1}e^{-z t}\,dt, \qquad z \in \C_+,
\end{equation}
is behind several basic properties of Pad\'e approximations to $e^{-z}.$
It can be found in e.g. \cite[p. 643]{Baker}, see also \cite[p. 880]{EgertR} and \cite[p. 3566]{Neubr}.
Note that 
\begin{equation}\label{qqq}
|Q_{n+1} (z)| \ge  1\qquad \text{and} \qquad \frac{|Q_{n+1}'(z)|}{|Q_{n+1}(z)|}\le 1, 
\end{equation}
for all $z \in \overline{\mathbb{C}}_{+}$ and $n \in \mathbb{N}$. For
$z =it, t \in \mathbb R$, the bounds in \eqref{qqq} were noted in
\cite[Propositions 1 and 2]{Neubr} and \cite[Lemma 3.1]{EgertR}. They extend
to $\mathbb{C}_{+}$ by the maximum principle applied to $1/Q_{n+1}$ and
$Q'_{n+1}/Q_{n+1}$. 
As noted in \cite[Lemma 3.2]{EgertR}, combining \eqref{AppDelta} with the first bound in \eqref{qqq}
yields 
\[
|\Delta _{r_{[n,n+1]}}(z)|\le \frac{1}{2} \left(\frac{n!}{(2n+1)!}\right)^2|z|^{2n+2},
\qquad z\in \mathbb{C}_{+}.
\]
Hence, invoking the simple inequality (\cite[Lemma 3.5]{EgertR}):
\begin{equation}
\label{inequal}
\frac{n!}{(2n+1)!}\le \frac{1}{(n+1)^{n+1}}, \qquad n \in \mathbb{N},
\end{equation}
we obtain
\begin{equation}
\label{EsR1}
|\Delta _{r_{[n,n+1]}}(z)|\le \frac{|z|^{2n+2}}{2(n+1)^{2(n+1)}},
\qquad z\in \mathbb{C}_{+}.
\end{equation}

To facilitate the subsequent estimates, we will need several additional properties of $\Delta_{r_{[n,n+1]}}$ given in the lemma below.

\begin{lemma}\label{Pade}
Let $r_{[n,n+1]}=P_n/Q_{n+1}, n \in \mathbb N,$ be the first subdiagonal Pad\'e approximations to $e^{-z},$
where $P_n$ and $Q_{n+1}$ are given by \eqref{polynom} and \eqref{polynom1}, respectively.
\begin{itemize}
\item [(i)] For all $n \in \N,$ 
\begin{equation}\label{EsR4}
\|\Delta_{r_{[n,n+1]}}\|_{{\rm H}^\infty(\C_+)}\le 2,\qquad \|\Delta_{r_{[n,n+1]}}'\|_{{\rm H}^\infty(\C_+)}\le 2.
\end{equation}
\item [(ii)] For all $n \in \N$ and $z \in \mathbb D^+_n,$
\begin{equation}\label{EsR2}
|\Delta_{r_{[n,n+1]}}'(z)|
\le 2\frac{|z|^{2n+1}}{(2n)!}
\int_0^1 (1-t)^n t^{n+1}e^{-t{\rm Re}\, z}\,dt.
\end{equation}
\end{itemize}

\end{lemma}
\begin{proof}
Fix $z \in \C_+$ and $n \in \N.$
To prove (i) it suffices to recall that by e.g. \cite[Theorem 4.12]{Hairer} the function $r_{[n,n+1]}$ is $\mathcal A$-stable,
and, moreover, by \cite[p. 334]{Iserles1},
its derivative $r_{[n,n+1]}'$ is $\mathcal A$-stable as well. 
(A weaker property $\|r_{[n,n+1]}'\|_{{\rm H}^\infty(\C_+)}\le 2$  was also noted in \cite[Proposition 2]{Neubr}.)

To deduce (ii), 
note that by (\ref{AppDelta}) and \eqref{qqq}, 
\begin{align*}
\Delta_{r_{[n,n+1]}}'(z)=&-\frac{Q_{n+1}'(z)}{Q^2_{n+1}(z)}\frac{z^{2n+2}}{(2n+1)!}
\int_0^1 (1-t)^n t^{n+1}e^{-z t}\,dt\\
+&\frac{1}{Q_{n+1}(z)}\frac{(2n+2)z^{2n+1}}{(2n+1)!}
\int_0^1 (1-t)^n t^{n+1}e^{-z t}\,dt\\
-&\frac{1}{Q_{n+1}(z)}\frac{z^{2n+2}}{(2n+1)!}
\int_0^1 (1-t)^n t^{n+2}e^{-z t}\,dt,
\end{align*}
and, letting $z=x+iy,$
\begin{align*}
|\Delta_{r_{[n,n+1]}}'(z)|
\le& \frac{|z|^{2n+2}}{(2n+1)!}
\int_0^1 (1-t)^n t^{n+1}e^{-x t}\,dt\\
+&\frac{(2n+2)|z|^{2n+1}}{(2n+1)!}
\int_0^1 (1-t)^n t^{n+1}e^{-x t}\,dt\\
+&\frac{|z|^{2n+2}}{(2n+1)!}
\int_0^1 (1-t)^n t^{n+2}e^{-x t}\,dt.
\end{align*}
Hence, if $z \in \mathbb D^+_n,$ then
\begin{equation*}
|\Delta_{r_{[n,n+1]}}'(z)|
\le 2\frac{|z|^{2n+1}}{(2n)!}
\int_0^1 (1-t)^n t^{n+1}e^{-x t}\,dt. \qedhere 
\end{equation*}
\end{proof}

For $s > 0$ and $n \in \mathbb N$ let now 
\[
\Delta_{r_{[n,n+1]}, s}(z):=z^{-s}\Delta_{r_{[n,n+1]}}(z), \qquad z \in \mathbb C_+.
\] 
The following  statement provides an estimate for $\|\Delta_{r_{[n,n+1]}, s}\|_{\mathcal{B}_0}$ leading
to the operator-norm error bound 
in Theorem \ref{mainER} via the $\mathcal B$-calculus. In view of possible improvements, we formulate the result
with an explicit constant.
\begin{lemma}\label{RozB}
For all $s>0$ and $n \in \mathbb N$ such that $n\ge (s-1)/2,$ one has 
\begin{equation}\label{BfR}
\|\Delta_{r_{[n,n+1]}, s}\|_{\mathcal{B}_0}\le 
\left\{\frac 12 +\frac{4}{s}
+\frac{s}{2}+\frac{4\pi s}{(s+1)\sin(\pi/(1+s))}
+4\log(2n+1)
\right\}\frac{1}{n^s}.
\end{equation}
\end{lemma}
\begin{proof}
Fix $n \in \mathbb N$ and $s \in (0,\infty)$ satisfying $2n+1\ge s>0.$
We write
\begin{equation}\label{AB}
\Delta_{r_{[n,n+1]}, s}'(z)=-s\frac{\Delta_{r_{[n,n+1]}}(z)}{z^{s+1}}+\frac{\Delta_{r_{[n,n+1]}}'(z)}{z^s},\qquad z\in \C_{+},
\end{equation}
and estimate the $\mathcal B_0$-norm of each of the terms in the right hand side of \eqref{AB} separately.

Let $z=x+iy\in \C_{+}.$
Then, observing that
 $|z|\ge n$ implies 
\[
2|z|^{s+1}\ge |z|^{s+1}+n^{s+1}\ge x^{s+1}+n^{s+1},
\]
and using \eqref{EsR1}, we have
\begin{equation}\label{first}
\begin{aligned}
&\int_0^\infty \sup_{y\in \R}\,\frac{|\Delta_{r_{[n,n+1]}}(z)|}{|z|^{s+1}}\,dx\\
\le& \int_0^\infty \sup_{y\in \R,\,|z|\le n}\,\frac{|\Delta_{r_{[n,n+1]}}(z)|}{|z|^{s+1}}\,dx
+\int_0^\infty \sup_{y\in \R,\,|z|\ge n}\,\frac{|\Delta_{r_{[n,n+1]}}(z)|}{|z|^{s+1}}\,dx
 \\
\le& \frac{1}{2n^{s+1}} \int_0^n \,dx
+ 4\int_0^\infty \frac{dx}{x^{s+1}+n^{s+1}}
\\
\le& \left(\frac{1}{2}
+\frac{4\pi}{(s+1)\sin(\pi/(1+s))}\right)n^{-s},
\end{aligned}
\end{equation}
since by \cite[p. 295]{Prud},
\[
\int_0^\infty \frac{dx}{x^{s+1}+1}
=\frac{\pi}{(s+1)\sin(\pi/(1+s))}.
\]
Next, in view of \eqref{EsR4}, we infer that
\[
|(\Delta_{r_{[n,n+1]}}(z))'|\le x^{-1} \|\Delta_{r_{[n,n+1]}}\|_{{\rm H}^\infty(\mathbb C_+)} \le \min\left\{2, x^{-1}\right\}\le
\frac{2}{x+1/2}
\]
for all $z \in \C_+.$  Hence, employing \eqref{EsR2},
\begin{equation}\label{labee}
\begin{aligned}
&\int_0^\infty \sup_{y\in \R}\,\frac{|\Delta_{r_{[n,n+1]}}'(z)|}{|z|^{s}}\,dx\\
\le& \int_0^\infty \sup_{y\in \R,\,|z|\le n}\,\frac{|\Delta_{r_{[n,n+1]}}'(z)|}{|z|^{s}}\,dx
+\int_0^\infty \sup_{y\in \R,\,|z|\ge n}\,\frac{|\Delta_{r_{[n,n+1]}}'(z)|}{|z|^{s}}\,dx
 \\
\le&\frac{2 n^{2n+1-s}}{(2n)!}
\int_0^1 (1-t)^n t^{n+1} \int_0^n  e^{-xt}   \,dx\,dt
+\int_0^\infty \frac{4\, dx}{(n^s+x^s)(1/2+x)}
 \\
\le&  \frac{2(n!)^2 n^{2n+1-s}}{(2n)!(2n+1)!} 
+\frac{4}{n^s}\int_0^\infty \frac{dx}{(1+x^s)(1/(2n)+x)},
\end{aligned}
\end{equation}
where we also used that
\[
\int_0^1(1-t)^nt^n\,dt=\frac{(n!)^2}{(2n+1)!}.
\]
Since
\begin{align*}
\int_0^\infty \frac{dx}{(1+x^s)(1/(2n)+x)}\le&\int_0^1\frac{dx}{x+1/(2n)}+
\int_1^\infty \frac{dx}{x^{s+1}}\\
=&\log(2n+1)+\frac{1}{s},
\end{align*}
and
\begin{equation}\label{labb}
\frac{2(n!)^2 n^{2n+1-s}}{(2n)!(2n+1)!} \le \frac{1}{2n^s},
\end{equation}
from \eqref{labee} it follows that
\begin{align}\label{second}
\int_0^\infty \sup_{y\in \R}\,\frac{|\Delta_{r_{[n,n+1]}}'(z)|}{|z|^{s}}\,dx
\le \left(\frac 12+\log(2n+1)+\frac{1}{s}\right)\frac{4}{n^s}. 
\end{align}
Now \eqref{AB}, \eqref{first} and \eqref{second} imply \eqref{BfR}.
\end{proof}

We do not know whether $\log(n+1)$ in \eqref{BfR} can be omitted.
This would follow from the estimate $\sup_{n \ge 1}\|r_{[n, n+1]}\|_{\mathcal B_0}<\infty.$
 Unfortunately, at the moment, this bound is out of reach as well. 

\section{The $\mathcal{B}$-calculus revisited and 
operator-norm estimates}\label{fcc}

\subsection{Construction of the $\mathcal B$-calculus}

As far as this paper relies on the $\mathcal B$-calculus
we provide  a streamlined approach  to its
construction, different from the one developed in \cite{BaGoTo} and \cite{BaGoTo1},
and thus make the paper essentially self-contained.
 In contrast to \cite{BaGoTo} and \cite{BaGoTo1},
the construction is based on the reproducing formula for the space $\mathcal B,$
and does not depend on any advanced function-theoretical machinery.

To start the construction, for $\lambda, z \in \mathbb C_+,$   write $\lambda=\alpha+i\beta$, 
and let
\[
K(z,\lambda):=-\frac{2}{\pi(z+\lambda)^2}, \qquad D_\lambda:=\frac{d}{d\lambda},
\qquad \text{and}\qquad
dS(\lambda):=\alpha \, d\beta \, d\alpha.
\]

Recall that
if $f\in \mathcal{B}$ then by \cite[Proposition 2.20]{BaGoTo}
$f$ can be recovered by the  reproducing formula 
\begin{equation}\label{A0}
f(z)=f(\infty)+\int_{\C_{+}} K(z,\overline{\lambda}) D_\lambda f(\lambda)\,dS(\lambda),\qquad z\in \C_{+},
\end{equation}
where the integral converges absolutely for every $z \in \C_+$ in view of  
\[
\sup_{\alpha >0} \alpha \int_{\mathbb R} |K(z,\alpha -i\beta)| \, d\beta <\infty 
\]
and Fubini's theorem.
The formula  \eqref{A0} was obtained initially in \cite[Proposition 2.20]{BaGoTo} in a complicated manner, and it was
reproved by elementary means in \cite[Theorem 2.3]{BaGoTo1} and also in \cite[Corollary 3.10 and Proposition 3.15]{BaGoTo2}.

Following an established route in 
the case of classical Riesz-Dunford calculus, given an operator $A$ on a Hilbert space $X,$
it is natural to  try to define $f(A)$ for $f \in \mathcal B$ by plugging
$A$ instead of independent variable $z$ into  \eqref{A0}, and thus creating an operator counterpart of \eqref{A0}.
In what follows, we formalise this procedure and show that it leads indeed to a well-defined functional calculus.
We say that an operator $A$ admits a $\mathcal B$-calculus $\Phi$ if $A$ is densely defined,
$\sigma(A)  \subset \overline{\mathbb C}_+,$ and there is a bounded algebra homomorphism $\Phi : \mathcal B \to
L(X)$ such that $\Phi((\lambda +\cdot)^{-1}) = (\lambda + A)^{-1}$ for all $\lambda \in \mathbb C_+.$
Note that if $\Phi$ is a $\mathcal B$-calculus for $A,$ then $\Phi(1)=I,$ see \cite[p. 33]{BaGoTo1}.
If $A$ admits a $\mathcal B$-calculus $\Phi,$ then we set $f(A):=\Phi(f), f \in \mathcal B.$

Note that if $-A$ generates a bounded $C_0$-semigroup $(e^{-tA})_{t \ge 0}$ on $X,$ with $\sup_{t \ge 0}\|{e^{-tA}}\|:=M,$ then 
Plancherel's theorem implies that for every $x \in X$  the $L^2(\R, X)$-norms of $\sqrt \alpha  
(\alpha+i\cdot+A)^{-1} x$ 
and $\sqrt \alpha \, (\alpha-i\cdot+A^*)^{-1} x$ are  
bounded by $\sqrt \pi M\|x\|$ uniformly in $\alpha >0.$ 
 Using this fact,
one infers easily that 
\begin{equation}\label{GSF}
\sup_{\alpha >0} \int_{\mathbb R}\alpha |\langle K(A,\alpha -i\beta) x, x^*\rangle| \, d \beta \le \pi M^2 \|x\|\|x^*\|
\end{equation}
for all $x, x^* \in X.$ (See \cite[Example 4.1]{BaGoTo} for a discussion of \eqref{GSF} 
and relevant references.)

From \eqref{GSF} it follows that the formula
\begin{equation}\label{OperA}
\langle f(A)x,x^*\rangle:=f(\infty)I+\int_{\C_{+}} \langle K(A,\overline{\lambda})x,x^*\rangle\,D_\lambda f(\lambda)\,dS(\lambda),  
\end{equation}
for all $x\in X$ and $x^*\in X,$ defines a bounded linear operator $f(A)$ on $X,$ and moreover 
\begin{equation}\label{b_est}
\|f(A)\|\le |f(\infty)|+\pi M^2\|f\|_{\mathcal B_0}\le \pi M^2 \|f\|_{\mathcal B}.
\end{equation}
Thus, we have a well-defined, bounded linear mapping
\begin{equation}\label{phia}
 \Phi_A:\mathcal B \to L(X),
\qquad \Phi_A(f):=f(A).
\end{equation}
 Moreover,   
as a simple calculation shows (see \cite[p. 42 and Lemma 4.2]{BaGoTo}),
 if $f=\widehat \nu$ with $\nu \in {\rm M}(\R_+),$
and  $f(A)$ is given by \eqref{OperA},
then
\[
f(A)x=\int_{0}^\infty e^{-tA}x\, d\nu(t), \qquad x \in X,
\]
so that $\Phi_A$ extends the HP-calculus, $\|\widehat \nu(A)\|\le \pi M^2 \|\widehat \nu\|_{\mathcal B},$ and  $\Phi_A((\cdot+\lambda)^{-1})=(\lambda+A)^{-1}, \lambda \in \C_+.$
 This fact was a starting point in \cite{BaGoTo}, and it will also be useful below.

If $\Phi_A$ is multiplicative, then  $\Phi_A$ is a functional calculus by the definition above. 
It would clearly coincide with the $\mathcal B$-calculus  constructed in \cite{BaGoTo}, since both calculi are given by the same formula \eqref{OperA}.
Having defined $f(A)$ by \eqref{OperA}, the argument in \cite{BaGoTo} proceeds with showing  the homomorphism
property of $\Phi_A.$ This is the main step of the construction, and it is rather involved.
Note that $\mathcal {LM}(\C_+)$ is not dense in $\mathcal B.$ So the approach in \cite{BaGoTo} relies on 
finding convenient dense sets $\mathcal G$ in $\mathcal B$ and using them to show that $\mathcal{LM}(\C_+)$ is dense in $\mathcal B$
in a suitable weak topology. Then the homomorphism property of $\Phi_A$ is established via several approximation unit arguments. The formula \eqref{A0} is not used explicitly, and $\mathcal G$  is created 
via Arveson's spectral theory for isometric groups applied to a $C_0$-group of vertical shifts on $\mathcal B.$

We give a simple alternative proof for the homomorphism property of $\Phi_A,$
where a mere validity of \eqref{A0} will allow us to set-up a functional calculus  for $A.$
The next identity, allowing one to separate the roles of $f$ and $g$ in  $(fg)(A),$ is the heart matter of our approach.
\begin{lemma}
Let $f,g\in \mathcal{B}_0.$ Then for all $x \in X$ and $x^* \in X,$ 
\begin{equation}\label{A2}
\langle (f g)(A)x,x^*\rangle =
\int_{\C_{+}} (D_{\mu}g)(\mu)\,
\langle [K(\cdot,\overline{\mu})f](A)x,x^*\rangle \,dS(\mu).
\end{equation}
\end{lemma}
\begin{proof}
Applying (\ref{OperA}) to $fg$ and  (\ref{A0}) to $g,$
and 
interchanging the integration order formally,
we infer that for all $x\in X$ and $x^*\in X,$ 
\begin{equation}\label{A1}
\begin{aligned}
&\langle(f g)(A)x,x^*\rangle
\\
=&\int_{\C_{+}} \langle K(A,\overline{\lambda})x,x^*\rangle D_\lambda\left(
f(\lambda)
\int_{\C_{+}}
K(\lambda,\overline{\mu})
D_\mu g(\mu)dS(\mu)\right)dS(\lambda)\\
=&\int_{\C_{+}} \langle K(A ,\overline{\lambda})x,x^*\rangle \left(
\int_{\C_{+}} D_\lambda(K(\lambda,\overline{\mu})f(\lambda))
D_\mu g(\mu)dS(\mu)\right)dS(\lambda)\\
=&\int_{\C_{+}}D_{\mu}g(\mu)
\int_{\C_{+}}
\langle K(A,\overline{\lambda})x,x^*\rangle
D_{\lambda}
\left(K(\lambda,\overline{\mu})f(\lambda)\right)\,dS(\lambda)\,dS(\mu).
\end{aligned}
\end{equation}
If the integrals in \eqref{A1} converge absolutely,
then Fubini's theorem makes the above argument rigorous, and the relation \eqref{A2} follows
from (\ref{A1}) and (\ref{OperA}).
Thus it remains to show the absolute convergence. 

To this aim, letting
$\mu=t+is$ with $t>0$ and $s \in \mathbb R,$
and assuming that $\|x\|=\|x^*\|=1,$ observe that 
\begin{equation}\label{ApA}
\begin{aligned}
\int_{\C_{+}}|D_{\mu}g(\mu)|&
\int_{\C_{+}}
|\langle K(A,\overline{\lambda})x,x^*\rangle|
|D_{\lambda}
\left(K(\lambda,\overline{\mu})f(\lambda)\right)|\,dS(\lambda)\,dS(\mu)\\
\le& \int_0^\infty \sup_{s\in \R}\,|g'(t+is)|
tL(t)\,dt, 
\end{aligned}
\end{equation}
where, for every $t > 0,$
\begin{align*}
L(t)
=&\int_{\R} \int_{\C_{+}} |\langle K(A,\overline{\lambda})x,x^*\rangle|
|D_{\lambda} \left(K(\lambda,t-is)f(\lambda)\right)|\,dS(\lambda)\,ds \\
=&\int_{\C_{+}} |\langle K(A,\overline{\lambda})x,x^*\rangle|
\left( \int_{\mathbb R}
|D_{\lambda}
\left(K(\lambda,t-is)f(\lambda)\right)|ds\right)\,dS(\lambda).
\end{align*}
Next, recalling that $\lambda=\alpha+i\beta,$ where $\alpha>0,\beta \in \mathbb R,$ 
we have 
\begin{align*}
&\int_{\R}
|D_{\lambda}
(K(\lambda,t-is)f(\lambda))|\,ds\\
&\le 2|f(\lambda)|\int_{\mathbb R}\frac{ds}{|\alpha+i\beta+t-is|^3}+
|f'(\lambda)|\int_{\mathbb R}\frac{ds}{|\alpha+i\beta+t-is|^2}\\
&\le 4\left(\frac{|f(\lambda)|}{(\alpha+t)^2}+
\frac{|f'(\lambda)|}{\alpha+t}\right), \qquad t >0.
\end{align*}

So, for all $t > 0,$
\begin{align*}
\frac{L(t)}{4}\le& \int_0^\infty \alpha
\sup_{\beta\in \R}\,\left(
\frac{|f(\lambda)|}{(\alpha+t)^2}
+\frac{|f'(\lambda)|}{\alpha+t}
\right)
\int_{\R}
|\langle K(A,\alpha-i\beta)x,x^*\rangle|
d\beta d\alpha\\
\le& \pi M^2
\int_0^\infty
\left(
\frac{1}{(\alpha+t)^2}\sup_{\beta\in\R}\,|f(\alpha+i\beta)|
+\frac{1}{\alpha+t}\sup_{\beta\in \R}\,|f'(\alpha+i\beta)|\right)
d\alpha\\
\le&t^{-1}
\pi M^2
\|f\|_{\mathcal{B}}.
\end{align*}
This yields the absolute convergence of the left-hand side of \eqref{ApA} and  finishes the proof.
\end{proof}

Now we are ready to prove the multiplicativity of $\Phi_A.$
\begin{prop}\label{LMA}
Let $f\in \mathcal{B}$ and $g\in \mathcal{B}$.
Then
\begin{equation}\label{MainA}
f(A)g(A)=(f g)(A).
\end{equation}
\end{prop}

\begin{proof}

Assume that $f, g \in \mathcal B_0.$ The proof of \eqref{MainA} will then be done in three steps.

{\bf a)} Let first $f\in \mathcal{LM}(\C_+)$ and $g\in \mathcal{B}_0$.
Since $\Phi_A$ extends the HP-calculus, we use the product rule for the HP-calculus to
obtain that
\[
\langle [K(\cdot,\overline{\mu})f](A)x,x^*\rangle=\langle K(A,\overline{\mu})f(A)x,x^*\rangle
\]
for all $\mu \in \mathbb C_+,$  $x \in X$ and $x^* \in X.$
Hence, by (\ref{A2}),
\begin{align*}
\langle(f g)(A)x,x^*\rangle
=&\int_{\C_{+}}(D_{\mu} g)(\mu)
\langle K(A,\overline{\mu})f(A)x,x^*\rangle\,dS(\mu)\\
=&\langle f(A)g(A)x,x^*\rangle, \qquad x \in X, \,\, x^* \in X,
\end{align*}
and (\ref{MainA}) follows.
Moreover, by symmetry,  (\ref{MainA})
holds if $g\in \mathcal{LM}(\C_+)$ and $f\in \mathcal{B}_0$.

{\bf b)} Let $\mu \in \C_+$ be fixed. By {\bf a)}, applying (\ref{MainA})  to
$f\in\mathcal{B}_0$ and 
$g=K(\cdot,\overline{\mu})\in \mathcal{LM}(\C_+),$
 we conclude that
\begin{equation}\label{RA0}
\langle [K(\cdot,\overline{\mu})f](A)x, x^*\rangle=\langle K(A,\overline{\mu})f(A)x,x^*\rangle
\end{equation}
for all $x \in X$ and $x^* \in X$.

{\bf c)} Let now $f,g\in \mathcal{B}_0$.
Then, using  (\ref{A2})  and (\ref{RA0}) and arguing as in {\bf a)}, 
we obtain (\ref{MainA}).

Finally, to deduce \eqref{MainA} for $f,g \in \mathcal B,$ it suffices to apply \eqref{MainA} to  
$f-f(\infty)$ and $g-g(\infty)$ from $\mathcal B_0.$ 
\end{proof}
Thus we arrive at a Hilbert space part of the result established in \cite[Theorem 4.4]{BaGoTo}.
\begin{cor}
If $-A$ is the generator of a bounded $C_0$-semigroup on a Hilbert space $X,$
then the mapping $\Phi _{A}: \mathcal B \to L(X)$,
given by \eqref{phia},
a functional calculus for $A,$ extending the HP-calculus. Moreover,
\eqref{b_est} holds.
\end{cor}
Since $\Phi_A$ coincides with the $\mathcal B$-calculus for $A$ constructed in \cite{BaGoTo},
we will refer to $\Phi_A$ as the $\mathcal B$-calculus in the sequel.
It is instructive to note that
if a linear operator  $A$ on $X$ admits a $\mathcal B$-calculus $\Phi,$
then $\Phi$ is unique by \cite[Theorem 6.2]{BaGoTo1}.
\begin{remark}
The construction of the $\mathcal B$-calculus given above can be adjusted to Banach spaces
for a wider class of $A$ satisfying the so-called GSF resolvent condition.
This goes however beyond the scope of the present paper.
Yet another simple and more powerful approach to construction of the $\mathcal B$-calculus
covering the case of arbitrary number of commuting Hilbert space semigroup generators
is developed in \cite{BaGoToB}.
\end{remark}
\subsection{Applications of the $\mathcal B$-calculus to rational approximations}

Now we turn to the proofs of the main results of this paper.
The arguments are straightforward and based on function-theoretical estimates obtained in Section \ref{fte}
and the $\mathcal B$-calculus.
We also rely on the fact that the $\mathcal B$-calculus (strictly) extends the HP-calculus.
While the HP-calculus will be used to keep a standard meaning for operator functions,
the $\mathcal B$-calculus will produce fine operator-norm estimates.
We employ the standard functional calculi theory, see e.g \cite[Chapter 1]{Haase},
and also \cite[Section 2.3]{EgertR} for a discussion close to our context.
 
The next lemma will be instrumental in the proof of Theorem \ref{mainBT}.
It shows, in particular, 
that estimating the $\mathcal B$-norms of $\Delta_{r, n, s}$ we, 
in fact, estimate the size of functions in $\mathcal{LM}$
in a finer way provided by the $\mathcal B$-norm. 
\begin{lemma}\label{deltalm}
Let $r$ be an $\mathcal A$-stable rational approximation of order $q$ to the exponential.
Then  for all $s \in [0, q+1]$ and $n \in \mathbb N$ one has
$\Delta_{r, n, s} \in \mathcal{LM}(\C_+).$ 
\end{lemma}
\begin{proof}
Let $n \in \mathbb N$ be fixed. Using a partial fraction expansion of $r,$   we infer that
$\Delta _{r, n} \in \mathcal {LM}(\mathbb{C}_{+}),$ and, since $\Delta_{r, n, 0}=\Delta_{r, n}$ on 
$\mathbb C_+,$ the statement holds for $s=0.$ 
Fix $s\in (0, q+1]$ and write
\begin{equation*}
\Delta _{r,n,s}(z)=\Delta _{r, n}(z)\eta _{s}(z) + (\eta _{s}(z-1)-
\eta _{s}(z))\Delta _{r, n}(z)
\end{equation*}
for all $z \in \mathbb{C}_{+}$. Recalling that
$\eta _{s}(z)=(1+z)^{-s}, z \in \mathbb{C}_{+}$, note that
$\eta _{s} \in \mathcal{LM}(\mathbb{C}_{+})$ (see e.g.
\cite[Lemma 3.3.4]{Haase}).
Hence,
$\Delta _{r, n}\eta _{s} \in \mathcal {LM}(\mathbb C_{+})$. Letting
\begin{equation*}
\nu _{s}(z):=\eta _{s}(z-1)-\eta _{s}(z), \qquad z \in \mathbb C_{+},
\end{equation*}
we prove next that
$\Delta _{r, n}\nu _{s} \in \mathcal{LM}(\mathbb{C}_{+})$, and thus
$\Delta _{r, n, s} \in \mathcal{LM}(\mathbb{C}_{+})$, as required. To this
end, by the mean value inequality for $\eta _{s}$, observe that
%
\begin{equation*}
|\nu _{s}(z)|\le \frac{s}{|z|^{s+1}} \qquad \text{and} \qquad |(\nu _{s}(z))'|
\le \frac{s(s+1)}{|z|^{s+2}}
\end{equation*}
for all $z \in \mathbb C_{+}$. Moreover, since $r$ is $\mathcal A$-stable,
decomposing $r$ into partial fractions we infer that
$\Delta _{r,n}' \in {\mathrm{H}}^{\infty}(\mathbb C_{+}).$
Hence,
\begin{equation}
|\left(\Delta_{r, n}\nu_s \right)'(z)|={\rm O}(|z|^{-(s+1)}) \qquad \text{as}\,\, |z|\to \infty, \, \, z \in \mathbb C_+.
\label{z0}
\end{equation}
Since $r$ approximates $e^{-z}$ with order $q,$ by Lemma~\ref{rates},(i), 
\begin{equation}
|\left(\Delta _{r, n}\nu_s \right)'(z)|={\rm O}\left(1 +(q+1-s)|z|^{q-s}\right) \qquad \text{as}\,\, |z| \to 0, \, \, z \in \mathbb C_+.
\label{z1}
\end{equation}
Thus, combining \eqref{z0} and \eqref{z1}, 
we conclude that
\begin{equation*}
\sup _{x>0} \|\left (\Delta _{r, n}\nu_s \right )'(x+i\cdot )\|_{L^{1}(
\mathbb R)}<\infty ,
\end{equation*}
i.e. $(\Delta _{r, n}\nu _{s} )'$ belongs to the Hardy space
${\mathrm{H}}^{1}(\mathbb C_{+})$. In view of \cite[Theorem 4.12]{BaGoTo2},
this implies that
$\Delta _{r, n} \nu _{s} \in \mathcal{LM}(\mathbb{C}_{+})$, and then
$\Delta _{r, n, s} \in \mathcal{LM}(\mathbb{C}_{+})$.
\end{proof}
\begin{remark}
There is an alternative approach to the proof of Lemma~\ref{deltalm}.
If $r$ an $\mathcal A$-stable rational approximation of order $q$ to the exponential, then as
above, by a partial fraction expansion, 
$\Delta _{r, n}\in \mathcal{LM}(\mathbb{C}_{+}).$
Moreover, 
$\Delta_{r, n, q+1} \in \mathcal{LM}(\mathbb{C}_{+})$ 
by, for instance, \cite[p. 685]{BT}. 
Then Lemma~\ref{deltalm} follows from an interpolation result in \cite[Corollary 4.3]{GT}. However, the
result from \cite{GT} is involved, and we preferred to give a direct, function-theoretical
argument.
\end{remark}

Now we are ready to prove Theorem \ref{mainBT}.
\smallskip

\begin{proof}[Proof of Theorem \ref{mainBT}]
Fix $n \in \N.$ 
Using a partial fraction expansion of $r,$ define $r^n(A/n)$ by the HP-calculus.
Since $-A$ generates a bounded $C_0$-semigroup on a Hilbert space, $A$ admits the $\mathcal B$-calculus.
So Theorem \ref{mainBT},(i) follows directly from Corollary \ref{MAC}, the invariance property \eqref{scal}, 
and the norm-bound \eqref{b_est} for the $\mathcal{B}$-calculus.
(Instead of invoking \eqref{scal} one may observe that $-tA/n$ generates a semigroup with the uniform norm bound $M$
for all $t \ge 0$.) 

To prove Theorem \ref{mainBT},(ii), fix $s \in (0, q+1]$ and $n \in \mathbb N.$
We define $\Delta_{r, n}(A)$ by the HP-calculus.  
If $h_s(z):= z^s, z \in \C_+,$ then in view of $\eta_{s} \in \mathcal{LM}(\C_+),$
we have $h_s \eta_{s+1}\in \mathcal {LM}(\C_+)$, see e.g. \cite[Lemma 3.3.1]{Haase} (or \cite[Lemma 4.1]{GT}). 
Thus $h_s$ belongs to the extended HP-calculus with the regulariser $\eta_{s+1}$.
By Lemma \ref{deltalm}, $\Delta_{r, n, s}$ belongs to the HP-calculus, 
and thus $\Delta_{r, n, s}(A)$ and $A^s$ are defined by the extended HP-calculus. 

Using the product rule for the (extended) $\mathcal B$-calculus
(see e.g. \cite[Proposition 1.2.2]{Haase}), we have
\begin{equation}\label{tone1}
\Delta_{r, n}(A)x=\Delta_{r, n, s}(A) A^sx, 
\qquad x \in {\rm dom}(A^s).
\end{equation}
In view of $\Delta_{r, n, s}(\infty)=0,$
Lemma \ref{T1P}
and \eqref{b_est} yield
\begin{equation}\label{tone}
\|\Delta_{r, n}(A)x\|\le \frac{C \pi M^2}{n^{qs/(q+1)}} \|A^s x\|
\end{equation}
for all $x \in {\rm dom}(A^s).$
Replacing $A$ by $tA$ in \eqref{tone},
we obtain \eqref{AGT}. 
\end{proof}

\begin{remark}
Arguing as in the proof of Theorem \ref{mainBT},(ii), and using \eqref{FolP2},
we conclude  that, under the assumptions of Theorem \ref{mainBT},
if  $s>0$ then there exists $C=C(s,r)>0$ such that 
\begin{equation}\label{A220A1}
\|r^n(tA/n)x\|\le
C \pi M^2\|(1+tA)^s x\|
\end{equation}
for all $x\in {\rm dom}\, (A^s)$ and $n \in \mathbb N.$
Thus the stability properties of $(r(tA/n))_{n \ge 1}$ improve for fixed $t>0$.
However, one can show easily that a version of \eqref{A220A1} with $tA$ replaced by $A$  in the left hand side of
\eqref{A220A1} does not, in general, hold.
\end{remark}

The proof of Theorem \ref{mainBTin} is similar to the proof Theorem \ref{mainBT},(i), although it relies
on stronger assumptions on $r.$

\medskip

\begin{proof}[Proof of Theorem \ref{mainBTin}] 
 Since $A$ admits the $\mathcal B$-calculus, the estimate \eqref{Ainf}
is a direct implication  of
\eqref{R2},  \eqref{scal} and  \eqref{b_est}.
As noted in the introduction, if \eqref{Ainf} holds, then \eqref{conve} follows from general theory.
Alternatively, using Theorem \ref{mainBT} (or Theorem \ref{brento}), we note that
$$
\lim_{n \to \infty}\Delta_{r, n}(tA)x=0, \qquad x \in {\rm dom} (A), \quad t \ge 0,
$$
and \eqref{conve} is then an immediate corollary of \eqref{Ainf}.  
\end{proof}

It is crucial to emphasize that
the estimates for approximation rates in Theorem \ref{mainBT},(ii) are sharp
in a sense clarified in Theorem \ref{CorMa} below.
To this end, recall that if $f \in \mathcal B$ (in particular, if $f \in \mathcal{LM}(\C_+)$)
and $-A$ generates a bounded $C_0$-semigroup on $X,$ then the next spectral inclusion theorem holds:
\begin{equation}\label{spmt}
\{f(\lambda): \lambda \in \sigma(A) \} \subset \sigma(f(A)).
\end{equation}
See e.g. \cite[Theorem 4.17, (3)]{BaGoTo} for \eqref{spmt} and more general statements.

\begin{thm}\label{CorMa}
Let $r$ be an $\mathcal A$-stable rational
approximation of order $q$ to the exponential, with $r^{(q+1)}(0)\neq (-1)^{q+1}.$
If $-A$ generates a bounded $C_0$-semigroup on a Hilbert space $X,$ and
$
i\mathbb R_+ \subset \sigma(A),
$
then there exist $C=C(r)>0$ and $n_0=n_0(r) \in \N$ such that
for every $s \in (0, q+1],$ 
\[
\|\Delta_{r, n, s}(tA) \|
\ge \frac{C t^s}{n^{sq/(q+1)}},\qquad n\ge n_0, \quad t \ge 0.
\]
\end{thm}

\begin{remark}\label{exact}
The assumption $r^{(q+1)}(0)\neq (-1)^{q+1}$ means that $r$ approximates $e^{-z}$ \emph{precisely} with order $q.$
\end{remark}
\begin{proof}
Let $s \in (0, q+1]$ be fixed, and, let $\delta$ and $a$ be defined as in \eqref{deff}.
From $r^{(q+1)}(0)\neq (-1)^{q+1},$ it follows that $a$ is different from zero.
Using Lemma \ref{rates},(ii) 
we infer that there exist  $R\in (0,1),$ $c>0,$ and a sequence of functions 
$(u_n)_{n \ge 1}$ such that $u_n \in {\rm Hol}\, (\mathbb D_{Rn})$ and  
\begin{equation*}
\Delta_{r,n,s}(z)=
\left(1-e^{az^{q+1}/n^q}e^{u_n(z)}\right){e^{-z}}{z^{-s}}, \qquad |u_n(z)|\le c\frac{|z|^{q+2}}{n^{q+1}},
\end{equation*}
for all $z\in \overline{\mathbb D}^+_{Rn^{\delta}}$ and $n \in \N.$

Choose $\gamma \in (0,1)$ satisfying also $\gamma \in (0,\min(R,\frac{2\pi}{|a|}))$
and set
\[
z_n=i \gamma n^\delta, \qquad n \in \mathbb N,
\]
so that $(z_n)_{n\ge 1} \subset \overline{\mathbb D}^+_{Rn^{\delta}}.$
If 
$
 b=a (i\gamma)^{q+1},
$
then
\[
|\Delta_{r,n,s}(z_n)|=
|1-e^{b} e^{u_n(i \gamma n^\delta)}|\gamma^{-s} n^{-\delta s}, \qquad n \in \N,
\]
and, taking into account \eqref{estim},
\begin{equation}\label{lower1}
\begin{aligned}
 n^{\delta s}|\Delta_{r, n, s}(z_n)|
\ge& |1-e^b|-e^b |u_n(i \gamma n^\delta)| e^{|u_n(i \gamma n^\delta)|}\\
=&|1-e^b|+\mbox{O}\left(\frac{1}{n^{1/(q+1)}}\right), \qquad n \to \infty,
\end{aligned}
\end{equation}
with $|1-e^b|\neq 0.$ 

Since $(z_n t^{-1})_{n \ge 1} \subset \sigma (A),$ using the spectral inclusion in
\eqref{spmt} we have
\begin{equation*}\label{lower}
\|\Delta_{r, n, s}(tA)\|\ge \sup_{z \in \sigma(A)} |\Delta_{r, n, s}(tz)|
\ge t^s|\Delta_{r, n, s}(z_n)|, \qquad n \in \N,
\end{equation*}
and the estimate \eqref{lower1} implies the claim.
\end{proof}
As a simple example satisfying the assumptions of Theorem \ref{CorMa},
one may consider a skew-adjoint operator $A$ with $\sigma(A)\supset i\mathbb R_+.$

We finish the section with indicating formally the argument for the proof of Theorem \ref{mainER}.
 
\smallskip

\begin{proof}[Proof of Theorem \ref{mainER}]
Fix $s>0$ and $n \ge (s-1)/2.$
The proof follows the same lines as the proof of Theorem \ref{mainBT},(ii),
with $\Delta_{r, n}$  and $\Delta_{r, n, s}$ replaced by $\Delta_{r_{[n,n+1]}}$ and $\Delta_{r_{[n, n+1]},s}$,
respectively.
By invoking the HP- and the $\mathcal B$- calculi,
we have the identity
\[
\Delta_{r_{[n,n+1]}}(A)x=\Delta_{r_{[n,n+1]}, s}(A) A^s x, \qquad x \in {\rm dom}(A^s).
\]
The statement is then a direct implication of Lemma \ref{RozB}, the fact that 
$\Delta_{r_{[n,n+1]}, s}(\infty)=0,$
 and \eqref{b_est}.
\end{proof}

\section{Appendix: Rates for rational approximation of exponential function}

Here we prove Lemma \ref{rates} formulated in Section \ref{fte}.
Its proof is simple but rather technical, and it is divided into several steps.

\begin{proof}[Proof of Lemma \ref{rates}.]
We will show that if 
$r$ is a rational approximation of order $q$ to the exponential, i.e. 
\begin{equation}\label{approq}
e^{-z}-r(z)=\mbox{O}(|z|^{q+1}), \qquad z\to 0,
\end{equation}
and $s \in [0,q+1]$ is fixed, then letting $\delta=q/(q+1)$ and 
\[
a=\left(r^{(q+1)}(0)-(-1)^{q+1}\right)((q+1)!)^{-1},
\]
there exist $R=R(r) \in (0,1)$ and 
$c=c(r)>0$ such that
\begin{equation}\label{main111}
|(\Delta_{r, n, s}(z))'|\le
\left(|a|(q+1-s)\frac{|z|^{q-s}}{n^q}
+\frac{c}{n^{\delta s}}\right) e^{-{\rm Re}\,z}
\end{equation}
for all $z \in \overline{\mathbb D}^+_{Rn^{\delta}}.$
This will provide the proof of Lemma \ref{rates},(i),
and on the way will prove Lemma \ref{rates},(ii) as well.
The proof will be done in several steps.

By \eqref{approq}, 
there is $R_0=R_0(r) \in (0,1)$ such that if  $\psi \in {\rm Hol}(\mathbb D_{R_0})$
is given by
\[
\psi(z):=-e^z\Delta_{r,1}(z), \qquad z \in \mathbb D_{R_0},
\] 
then
\begin{equation}\label{D10}
r(z)=e^{-z}(1+\psi(z)) 
\end{equation}
with
\begin{equation}
\label{CGG}
|\psi(z)|\le b|z|^{q+1},
\end{equation}
and (as consequence)
\begin{equation}\label{CGG1}
|\psi'(z)|\le b_1|z|^q,
\end{equation}
for all $z \in \mathbb D_{R_0}$ and some positive $b=b(r)$ and $b_1=b_1(r).$
Observing that for all $z \in \mathbb D_{R_0},$
\[
\psi^{(q+1)}(z)=-e^z\sum_{k=0}^{q+1}\binom {q+1}{k} \Delta^{(k)}_{r,1}(z),
\]
\noindent
and $\Delta^{(k)}_{r,1}(0)=0$ for all $0\le  k \le q,$  we infer that
\[
\psi^{(q+1)}(0)
=-\Delta^{(q+1)}_{r,1}(0)=r^{(q+1)}(0)-(-1)^{q+1},
\]
\noindent
and thus
\[
\psi(z)=\frac{r^{(q+1)}(0)-(-1)^{q+1}}{(q+1)!}z^{q+1}+
\mbox{O}(z^{q+2}),\quad z\to 0.
\]
\noindent
So, we can write
\begin{equation}\label{D1}
 \psi(z)=az^{q+1}+m(z),
\end{equation}
\noindent
 where $m \in {\rm Hol}(\mathbb D_{R_0})$ and
\begin{equation}\label{m}
|m(z)|\le \tilde b|z|^{q+2},\qquad |m'(z)|\le \tilde b_1|z|^{q+1},
\qquad z\in \mathbb D_{R_0},
\end{equation}
for some positive $\tilde b=\tilde b(r)$ and $\tilde b_1=\tilde b_1(r).$ 

To simplify further presentation, given $\alpha \in (0, R]$ and  positive functions $f$ and $g$ on $\mathbb D_{R}$
we use the notation $f(z)\lesssim g(z)$ for all $z \in \mathbb D_\alpha$ if $f(z)\leq C g(z), z \in \mathbb D_\alpha,$ where a constant $C > 0$
depends only on parameters associated to $r$ (and independent of $z$).

Observe that from \eqref{CGG} it follows that there exists $R \in (0, R_0)$ satisfying
\begin{equation}\label{12}
|\psi(z)|\le \frac{1}{2},\qquad z\in \mathbb D_{R}.
\end{equation}
\noindent
Using the elementary inequality
\[
|\log(1+z)-z|\le |z|^2,\qquad |z|\le 1/2,
\]
the bound (\ref{12}) and taking into account \eqref{m} and \eqref{CGG}, we infer that for all 
$z\in \mathbb D_{R},$
\begin{align*}
|\log(1+\psi(z))-a z^{q+1}|
\le& |m(z)|+|\psi(z)|^2\\
\lesssim & |z|^{q+2}+|z|^{2(q+1)}
\lesssim |z|^{q+2}.
\end{align*}
\noindent
Hence, if $z\in \mathbb D_{R}$ then e.g. by power series expansion
 there is $m_0 \in {\rm Hol}(\mathbb D_{R})$
such that
\begin{equation}\label{Alog}
\log(1+\psi(z))=az^{q+1}+m_0(z),\qquad
|m_0(z)|\lesssim |z|^{q+2},
\end{equation}
and, in view of \eqref{D10},
\begin{equation}\label{expAA}
r(z)=e^{-z}e^{az^{q+1}}e^{m_0(z)}.
\end{equation}
\noindent
Thus if $n \in \mathbb N$ is fixed, and
\[
r_n(z)=r^n(z/n), \qquad z \in \mathbb D_{Rn},
\]
then, by (\ref{expAA}), we have
\begin{equation}\label{ABC}
r_n(z)=e^{-z}e^{az^{q+1}/n^q}e^{m_{0,n}(z)},
\end{equation}
where  $m_{0,n} \in {\rm Hol}(\mathbb D_{Rn})$ is defined by
$m_{0,n}(z):=n \cdot m_{0}(z/n), z \in \mathbb D_{Rn},$
and satisfies
\begin{equation}\label{ABCC} 
|m_{0,n}(z)|
\lesssim \frac{|z|^{q+2}}{n^{q+1}}, \qquad z \in \mathbb D_{Rn^\delta}.
\end{equation}
Letting $u_n(z):=m_{0,n}(z), z \in \mathbb D_{Rn^\delta},$ and noting that $\Delta_{r,n, s}$ extends continuously to $[-iRn^\delta, iRn^\delta],$ we obtain Lemma \ref{rates},(ii).

Next, using the estimates
\begin{equation}\label{estim}
|e^z-1|\le |z|e^{|z|},\qquad
|e^z-1-z|\le \frac{|z|^2}{2}e^{|z|}, \qquad z \in \mathbb C,
\end{equation}
\noindent
and \eqref{ABC}, we define $m_{1,n} \in{\rm Hol}(\mathbb D_{Rn})$ by
\begin{equation}\label{phi0}
r_n(z)=(1+m_{1,n}(z))e^{-z},
\end{equation}
\noindent
and note that
\begin{equation}\label{m1n}
|m_{1,n}(z)|\le
(|az^{q+1}/{n^q}|+|m_{0,n}(z)|)e^{|az^{q+1}/{n^q}|+|m_{0,n}(z)|}\lesssim \frac{|z|^{q+1}}{n^q}
\end{equation}
for all $z \in \mathbb D_{Rn^\delta}.$
Similarly, if  $m_{2,n} \in {\rm Hol}(\mathbb D_{Rn})$ is given by 
\begin{equation}\label{phi}
r_n(z)=\left(1+a\frac{z^{q+1}}{n^q}+m_{2,n}(z)\right)e^{-z},
\end{equation}
\noindent
then for every $z \in \mathbb D_{Rn^\delta},$
\begin{align*}
|m_{2,n}(z)|
\le& |m_{0,n}(z)|+\left(|az^{q+1}/n^q|^2+|m_{0,n}(z)|^2\right)
e^{|az^{q+1}/{n^q}|+|m_{0,n}(z)|}\\
\lesssim &  \frac{|z|^{q+2}}{n^{q+1}}+
\frac{|z|^{2(q+1)}}{n^{2q}}+
\frac{|z|^{2(q+2)}}{n^{2(q+1)}}
\lesssim   \frac{|z|^{q+2}}{n^q}.
\end{align*}

Furthermore, recalling that
$
\Delta_{r,n}(z)=e^{-z}-r_n(z),
$
and using \eqref{phi0} and \eqref{phi}, note that
\begin{equation}\label{FG0}
\Delta_{r,n}(z)=-e^{-z}m_{1,n}(z)
\end{equation}
\noindent 
and
\begin{equation}\label{FG}
\Delta_{r,n}(z)=-e^{-z}\left(a\frac{z^{q+1}}{n^q}+m_{2,n}(z)
\right).
\end{equation}
\noindent
for all $z \in \mathbb D_{Rn}.$ Since by \eqref{D10},
\[
\frac{r'(z)}{r(z)}=-1+\frac{\psi'(z)}{1+\psi(z)},
\]
\noindent
we have
\begin{equation}\label{rnn}
\frac{r_n'(z)}{r_n(z)}
=-1+\psi'(z/n)\left(1-\frac{\psi(z/n)}{1+\psi(z/n)}\right), \qquad z \in \mathbb D_{Rn}.
\end{equation}
\noindent
Combining  \eqref{D1}, (\ref{phi0}), (\ref{FG0}) and \eqref{rnn}, we obtain that
\begin{equation}\label{MF}
\begin{aligned}
&(\Delta_{r,n}(z))'\\
=&-\Delta_{r,n}(z)-r_n(z)\psi'(z/n)\left(1-\frac{\psi(z/n)}{1+\psi(z/n)}\right)
\\
=& m_{1,n}(z)e^{-z}
-(1+m_{1,n}(z))
\left(a(q+1)\frac{z^q}{n^q}+m'(z/n)\right)\left(1-\frac{\psi(z/n)}{1+\psi(z/n)}\right)e^{-z}
\\
=&-a(q+1)\frac{z^q}{n^q}e^{-z}+m_{3,n}(z)e^{-z},
\end{aligned}
\end{equation}
\noindent
where $m_{3,n} \in {\rm Hol}(\mathbb D_{Rn})$ is defined by
\begin{align*}
m_{3,n}(z):=& m_{1,n}(z)-m_{1,n}(z)
\left(a(q+1)\frac{z^q}{n^q}+m'(z/n)\right)\frac{1}{1+\psi(z/n)}\\
+&\left(a(q+1)\frac{z^q}{n^q}+m'(z/n)\right)\frac{\psi(z/n)}{1+\psi(z/n)}
-m'(z/n), \qquad z \in \mathbb D_{Rn}.
\end{align*}
\noindent
Taking into account $|z|\lesssim n$, $z\in \mathbb D_{R n^\delta}$,
and employing \eqref{m}, \eqref{CGG}, \eqref{12}, and \eqref{m1n}, we have 
\begin{equation}\label{MF1}
\begin{aligned}
|m_{3,n}(z)|\lesssim& \frac{|z|^{q+1}}{n^q}\left(
1+\frac{|z|^q}{n^q}+\frac{|z|^{q+1}}{n^{q+1}}\right)\\
+&\left(\frac{|z|^q}{n^q}+\frac{|z|^{q+1}}{n^{q+1}}\right)\frac{|z|^{q+1}}{n^{q+1}}+\frac{|z|^{q+1}}{n^{q+1}}
\\
\lesssim& \frac{|z|^{q+1}}{n^q}, \qquad z \in \mathbb D_{Rn^\delta}.
\end{aligned}
\end{equation}

Therefore, using  \eqref{FG}, \eqref{MF}, and \eqref{MF1}, we infer 
that for every $z \in \mathbb D^+_{Rn},$
\begin{equation}\label{final}
\begin{aligned}
(\Delta_{r, n, s}(z))'=&\frac{\Delta_{r,n}(z)'}{z^s}-
s\frac{\Delta_{r,n}(z)}{z^{s+1}}\\
=&a(s-(q+1))\frac{z^{q-s}}{n^q}e^{-z}
+\frac{m_{3,n}(z)}{z^s}e^{-z}
+m_{2,n}(z)\frac{s}{z^{s+1}}e^{-z}
 \\
=&a(s-(q+1))\frac{z^{q-s}}{n^q}e^{-z}+m_{4,n}(z)e^{-z},
\end{aligned}
\end{equation}
\noindent
with $m_{4,n} \in {\rm Hol}(\mathbb D^+_{Rn})$ satisfying
\begin{equation}\label{m4}
|m_{4,n}(z)|\lesssim \frac{|z|^{q+1-s}}{n^q}, \qquad z \in \mathbb D^+_{Rn^\delta}.
\end{equation}

Recalling the relation $(q+1)\delta=q$, observe that
\begin{equation}\label{fract}
\frac{|z|^{q+1-s}}{n^q}\le R^{q+1-s}\frac{n^{\delta(q+1-s)}}{n^q}=\frac{R^{q+1-s}}{n^{\delta s}}.
\end{equation}
Hence, 
\eqref{final}, \eqref{m4} and \eqref{fract} imply \eqref{main111}.
\end{proof}

\section{Acknowledgment}

We would like to thank the referee for various useful suggestions and remarks
that led to improvement of the paper.

\end{document}